\documentclass[11pt]{article}

\usepackage[a4paper]{geometry}
\usepackage{lmodern}        
\usepackage{graphicx,color}
\usepackage{xcolor}
\usepackage{enumitem}       
\usepackage{multirow}       
\usepackage{amsmath,amssymb,amsfonts,amsthm,mathtools,bm}
\usepackage{dsfont}
\usepackage{bbm}
\usepackage{comment}
\usepackage{a4wide}

\usepackage{hyperref}
\usepackage{cleveref}

\usepackage{tikz}
\usepackage{pgfplots}
\usepgfplotslibrary{groupplots}
\usetikzlibrary{arrows.meta,positioning,calc,decorations.pathreplacing}
\pgfplotsset{compat=1.18}

\newtheorem{proposition}{Proposition}[section]
\newtheorem{theorem}[proposition]{Theorem}
\newtheorem{lemma}[proposition]{Lemma}
\newtheorem{corollary}[proposition]{Corollary}

\theoremstyle{definition}
\newtheorem{remark}[proposition]{Remark}

\newtheorem{question}[proposition]{Question}

\theoremstyle{remark}

\newenvironment{proofof}[1]{\smallskip\noindent{\textbf{Proof~of~#1.}}%
  \hspace{1pt}}{\hspace{-5pt}{\nobreak\quad\nobreak\hfill\nobreak%
    $\square$\vspace{2pt}\par}\smallskip\goodbreak}

\numberwithin{equation}{section}
\numberwithin{figure}{section}
\numberwithin{table}{section}

\crefname{theorem}{theorem}{theorems}
\Crefname{theorem}{Theorem}{Theorems}

\crefname{proposition}{proposition}{propositions}
\Crefname{proposition}{Proposition}{Propositions}

\crefname{lemma}{lemma}{lemmas}
\Crefname{lemma}{Lemma}{Lemmas}

\crefname{corollary}{corollary}{corollaries}
\Crefname{corollary}{Corollary}{Corollaries}

\crefname{definition}{definition}{definitions}
\Crefname{definition}{Definition}{Definitions}

\crefname{remark}{remark}{remarks}
\Crefname{remark}{Remark}{Remarks}

\crefname{example}{example}{examples}
\Crefname{example}{Example}{Examples}

\Crefname{question}{Question}{Question}

\Crefname{assumptions}{Assumptions}{assumptions}

\allowdisplaybreaks[4]
\newcommand{\RR}{\mathbb{R}}
\newcommand{\Ss}{\mathbb{S}}
\newcommand{\ve}{\varepsilon}

\begin{document}
\title{Dense orbits for scale-invariant rotationally symmetric solutions of the 2D Euler equations}
\author{Ibrahim Suleiman\footnote{Courant Institute School of Mathematics, Computing and Data Science, New York University, 251 Mercer Street, New York, NY 10012, USA, \texttt{is1647@nyu.edu}}}
\maketitle

 \begin{abstract}
Let $m \ge 4$. We prove that there is a forward orbit of the 2D Euler system for scale-invariant $m$-fold symmetric vorticities, namely 
\begin{equation*}
    \begin{aligned}
                \partial_tg + 2G\partial_\theta g &= 0 \\
                4G + \partial_{\theta\theta}G &= g,
    \end{aligned}
\end{equation*}
which is dense in $S :=\{g \in L_{m,\mathrm{odd}}^\infty(\Ss^1): g\sin(m\theta) \ge 0 \text{ and }\|g\|_{L^\infty} \le 1\}$ equipped with the topology of weak$^*$ convergence on $L^\infty(\Ss^1)$, where $L_{m,\mathrm{odd}}^\infty(\Ss^1)$ denotes the space of odd and $m$-fold symmetric functions in $L^\infty(\Ss^1)$. The system above was first derived by Elgindi and Jeong, who showed well-posedness in the $m$-fold symmetric class. In subsequent work by Elgindi, Murray and Said, it was shown that solutions with regulated vorticity relax to piece-wise constant steady states with finitely many jumps; our construction shows that this regularity assumption cannot be relaxed. The proof starts from a simple linear mechanism: by evolving prescribed step functions backwards from chosen times and superimposing the resulting data, one can arrange for a single linear orbit to approximate a countable dense family. We show that this mechanism persists for the nonlinear Euler coupling. The key stability estimates imply that the error introduced at each stage becomes arbitrarily small when the next approximation time is taken sufficiently far in the future, allowing the construction to be iterated indefinitely.
\end{abstract}

\section{Introduction}
This paper considers the system 
\begin{equation}
    \label{eqn:Euler:S1:transport}
    \begin{aligned}
                \partial_tg + 2G\partial_\theta g &= 0 \\
                g(0, \theta) &= g_0(\theta)
    \end{aligned}
\end{equation}
on $\mathbb{S}^1 = \{ \theta: -\pi \le \theta < \pi \}$, where $G(t,\cdot)$ is the unique solution to the equation
\begin{equation}
    \label{eqn:Euler:S1:BSlaw}
\begin{aligned}
        4G + \partial_{\theta\theta} G &= g\\
    \frac{1}{2\pi}\int_{-\pi}^\pi G(t,\theta)e^{\pm 2i\theta} \ \mathrm{d} \theta &= 0
\end{aligned}
\end{equation}
Let $L_m^p(\Ss^1)$ denote the space of $m$-fold symmetric functions on $\Ss^1$ which belong to $L^p(\Ss^1)$, where $1 \le p \le \infty$. If $g_0 \in L_m^\infty(\Ss^1)$  for some $m \ge 3$, the system \eqref{eqn:Euler:S1:transport}-\eqref{eqn:Euler:S1:BSlaw} admits a global in time solution $g \in L^\infty(\RR; L_m^\infty(\Ss^1)) \cap C(\RR; L_m^1(\Ss^1))$. This is due to Elgindi and Jeong \cite{elgindi2020symmetries}.

Let $L_{m,\mathrm{odd}}^\infty(\Ss^1)$ denote the space of odd and $m$-fold symmetric functions in $L^\infty(\Ss^1)$. Our main theorem is the following.
\begin{theorem}
\label[theorem]{thm:main}
    Let $m \ge 4$. There exists initial data $g_0 \in L^\infty_{m,\mathrm{odd}}(\Ss^1)$ such that the weak$^*$ $\omega_+$-limit set of $g_0$ is
    \begin{equation*}
       \omega_+(g_0) = \{g \in L_{m,\mathrm{odd}}^\infty(\Ss^1): g\sin(m\theta) \ge 0 \text{ and }\|g\|_{L^\infty} \le \|g_0\|_{L^\infty}\}
    \end{equation*}
\end{theorem}

\subsection{Historical remarks}
Global existence of solutions to the 2D Euler equations (written here in vorticity form)
\begin{equation}
    \begin{aligned}
        \partial_t \omega + u\cdot \nabla \omega &= 0 \\
        u &= \nabla^\perp(-\Delta)^{-1}\omega
    \end{aligned}
\end{equation}
is well known in various settings and even under weak regularity assumptions on the initial data. However, comparatively little is known about the qualitative behavior of these solutions as $t \to \infty$, especially away from equilibrium. The breakthrough work of Bedrossian and Masmoudi \cite{bedrossian2015inviscid}, the work of Ionescu and Jia \cite{ionescu2022axi} and their extensions \cite{masmoudi2024nonlinear},\cite{ionescu2023nonlinear} are examples of settings where long-time behavior is described near equilibria. We refer the reader to the review article \cite{drivas2023singularity} and references therein for a discussion of long-time behavior for the 2D Euler equations. The setting of this paper, described below, is non-perturbative.

The system \eqref{eqn:Euler:S1:transport}-\eqref{eqn:Euler:S1:BSlaw} was first introduced by Elgindi and Jeong in \cite{elgindi2020symmetries}, where they developed a well-posedness theory for the 2D Euler equations in the class of $m$-fold rotationally symmetric vorticities when $m \ge 3$. For the distinguished sub-class of scale-invariant and $m$-fold symmetric vorticities, the 2D Euler equations reduce to  \eqref{eqn:Euler:S1:transport}-\eqref{eqn:Euler:S1:BSlaw}. This provided a setting away from equilibrium where detailed qualitative information could be obtained in the limit $t \to \infty$, as was done in \cite{elgindi2022long}. Indeed, it was shown that for $m \ge 4$, solutions of \eqref{eqn:Euler:S1:transport}-\eqref{eqn:Euler:S1:BSlaw} with $g_0 \in \operatorname{Reg}(\Ss^1)$\footnote{The space of regulated functions $\operatorname{Reg}(\Ss^1)$ is by definition the Banach space of bounded functions on $\Ss^1$ which possess left and right limits at every point.} relax to steady states which are piecewise constant with finitely many jumps. This left open the fate of solutions with $g_0 \notin \operatorname{Reg}(\Ss^1)$, which the present paper addresses. Our main result exhibits a solution showing that the hypothesis $g_0 \in \operatorname{Reg}(\Ss^1)$ cannot be removed if one allows arbitrary $L^\infty$ data.  We mention also the very recent result of \cite{cao2026gradient} which treats the case $m=3$ left open in \cite{elgindi2022long}.

\subsection{Idea of the construction}
 Let $g_0$ be odd about $0$ and non-negative on $[0, \frac{\pi}{m}]$ (i.e. $g_0\sin(m\theta) \ge 0$ on $\Ss^1$) and let $g(t, \cdot) = g(t)$ be the corresponding solution of \eqref{eqn:Euler:S1:transport}. First note that by uniqueness, $g(t, \cdot)$ is odd for all $t$, as is $G(t, \cdot)$. Since $G(t, -\pi/m) = G(t, \pi/m)$, we have $G(t, 0) = G(t, \pi/m) = 0$. Thus we may restrict our attention to $[0, \frac{\pi}{m}]$ and extend to $\Ss^1$ by odd symmetry and $m$-fold symmetry afterwards. The reduced system is now

\begin{equation}
    \label{eqn:Euler:reduced:transport}
    \begin{aligned}
                \partial_tg + 2G\partial_\theta g &= 0 \\
                g(0, \theta) &= g_0(\theta)
    \end{aligned}
\end{equation}
on $[0, \frac{\pi}{m}]$, where $G(t,\cdot)$ is the unique solution to the boundary value problem
\begin{equation}
    \label{eqn:Euler:reduced:BSlaw}
\begin{aligned}
        4G + \partial_{\theta\theta} G &= g\\
        G(t, 0)=G(t, \pi/m) &= 0
\end{aligned}
\end{equation}
Since $g$ is transported and $g_0 \ge 0$ on $[0, \frac{\pi}{m}]$, the same is true for $g(t, \cdot)$. 

We can express $G(t, \theta)$ by means of a Green's kernel
\begin{equation}
\label{eqn:G:kernel:expr}
        G(t, \theta) = \int_0^{\frac{\pi}{m}}K_m(\theta, \zeta)g(t,\zeta) \ \mathrm{d}\zeta,
    \end{equation}
where
\begin{equation}
        K_m(\theta, \zeta) = -\frac{1}{2s_m}\begin{cases}
            \sin(2\theta)\sin(2(\frac{\pi}{m}-\zeta)),\quad 0\le \theta \le \zeta \\
             \sin(2\zeta)\sin(2(\frac{\pi}{m}-\theta)),\quad  \zeta \le \theta \le \frac{\pi}{m}
        \end{cases}
\end{equation}
and
\begin{equation}
    s_m := \sin(2\pi/m)
\end{equation}
Notice that $K_m(\theta, \zeta) < 0$ for $(\theta,\zeta) \in (0, \frac{\pi}{m})^2$. Hence 
\begin{equation}
\label{eqn:directionality}
    \text{$G \le 0$ for all $t$}
\end{equation}
In addition, for non-trivial $g_0$, it is simple to check that $\partial_\theta G(t, \theta)$ is monotone increasing for each $t$, that $\partial_\theta G(t, 0) < 0$ while $\partial_\theta G(t, \frac{\pi}{m}) > 0$, and that $G(t,\cdot)$ has a unique minimum in $(0, \frac{\pi}{m})$ at each $t$. This says that the velocity $G$ is expansive on the right side of the domain, and contractive on the left.

The foregoing facts about $G$ will play a central role in our construction. Indeed, the point of departure for the proof of \Cref{thm:main} is the following observation.
Consider the data 
\begin{equation*}
    g_0(\theta) = c\mathbf{1}_{(a, b]}(\theta),
\end{equation*}
where $0 < a < b < \frac{\pi}{m}$ and $0 < c \le 1$. Under the evolution \eqref{eqn:Euler:reduced:transport}-\eqref{eqn:Euler:reduced:BSlaw}, $g(t, \theta)$ remains a characteristic function for all $t$:
\begin{equation*}
    g(t,\theta) = c\mathbf{1}_{(a(t), b(t)]}(\theta)
\end{equation*}
We can therefore reduce the dynamics to a system of ODEs for the endpoints $a(t)$ and $b(t)$. Using $a' = 2G(t,a), b' = 2G(t, b)$ and \eqref{eqn:G:kernel:expr}, we derive
\begin{equation}
\label{eqn:ode:simple}
    \begin{aligned}
        a'(t) &= -\frac{c}{2s_m}\sin(2a(t))\{\cos(2({\textstyle\frac{\pi}{m}}-b(t))) - \cos(2({\textstyle\frac{\pi}{m}}-a(t)))\} \\
        b'(t) &= \frac{c}{2s_m}\sin(2({\textstyle\frac{\pi}{m}}-b(t)))\{\cos(2b(t)) - \cos(2a(t))\}\\
            a(0) &= a, \quad b(0) = b
    \end{aligned}
\end{equation}
Since $a', b' \le 0$, the limits $\lim_{t\to\infty} a(t), \lim_{t\to\infty} b(t)$ exist. It follows from \eqref{eqn:ode:simple} that $\lim_{t\to\infty} a'(t)= \lim_{t\to\infty} b'(t) = 0$ and consequently \begin{equation*}
    c_\infty:=\lim_{t\to\infty} a(t)=\lim_{t\to\infty} b(t)
\end{equation*}
We claim that 
\begin{equation}
    \label{claim:cinf=0}
    c_\infty = 0.
\end{equation}
The claim \eqref{claim:cinf=0} is an immediate consequence of \Cref{lem:segment:quant} below. Its significance is that any segment (represented by $c\mathbf{1}_{(a, b]}(\theta)$) originating in $(0, \frac{\pi}{m})$ travels left-wards and is attracted towards $\theta = 0$ in infinite time. If this segment were initially localized near $\frac{\pi}{m}$, it is initially in the expansive region, and so it expands while traveling into the bulk, and afterwards begins to contract as it ultimately travels to the origin (see Figure~\ref{fig:single-packet-snapshots} below). 

Since countably many linear combinations of segments, i.e. step functions, can be used to approximate any target function (in an appropriate topology), we can envision stacking an enumeration of these step functions right-wards at ever decreasing scales near $\frac{\pi}{m}$  at the initial time, and using the properties of $G$ to unwind and collapse the stack leftwards, one step function at a time. This allows us to reach every neighborhood of any given function in the desired set.
The construction consists in making the foregoing heuristic considerations rigorous. The main issue to get around is non-linearity. Indeed, for a linear model such as
\begin{equation*}
    \partial_t f -\sin(m\theta)\partial_\theta f = 0,
\end{equation*}
it is relatively easy to construct $g_0$ with the desired properties by time reversibility and superposition. Under the non-linear dynamics of \eqref{eqn:Euler:S1:transport}-\eqref{eqn:Euler:S1:BSlaw}, combinations interact and we must prove that these interactions are under control. Fortunately, the system gives sufficient decay of disturbances in the interior of $[0, \frac{\pi}{m}]$ that enables us to close estimates. 

\section{Evolution of step functions}

In this section, we study the evolution of a step function supported in $(0, \frac{\pi}{m})$ and deduce some quantitative statements. Before proceeding, it is instructive to expound the case of a single interval considered in the previous section. We will prove in particular \eqref{claim:cinf=0} below and by so doing exhibit additional quantitative information.   

We begin with the following useful statement.
\begin{lemma}
\label[lemma]{lem:G-monotonicity}
Assume that
\begin{equation*}
    4G  + \partial_{\theta\theta} G = g, \quad G(0) = 0,
\end{equation*}
where $g \ge 0$. Then the function
    \begin{equation*}
    \frac{G}{\sin(2\theta)}
    \end{equation*}
    is monotonically increasing on $(0, \frac{\pi}{m}]$.
\end{lemma}
\begin{proof}
    \begin{equation*}
        \begin{aligned}
            \partial_\theta \left(\frac{G}{\sin(2\theta)}\right) &= \frac{\partial_\theta G\sin(2\theta) - 2G\cos(2\theta)}{\sin^2(2\theta)}
        \end{aligned}
    \end{equation*}
   The numerator is equal to $0$ at $\theta = 0$ and its derivative is
    \begin{equation*}
        \begin{aligned}
            (\partial_{\theta\theta} G  + 4 G)\sin(2\theta) = g\sin(2\theta) \ge 0.
        \end{aligned}
    \end{equation*}
    This implies that
    \begin{equation*}
         \partial_\theta \left(\frac{G}{\sin(2\theta)}\right) \ge 0.
    \end{equation*}
\end{proof}

\subsection{Evolution of a single step}

\begin{lemma}
\label[lemma]{lem:segment:quant}
    Assume that $a(t), b(t)$ satisfy \eqref{eqn:ode:simple}. Then
    \begin{equation*}
        \begin{aligned}
           \frac{1}{1+t}\lesssim b(t) &\lesssim \frac{1}{1+t}\\
            \lim_{t\to\infty} \frac{a(t)}{b(t)} &= 0.
        \end{aligned}
    \end{equation*}
More precisely,
\begin{equation*}
    \frac{a(t)}{b(t)} \lesssim (1+t)^{-\gamma}
\end{equation*}
for some $\gamma > 0$.
Here, the implicit constants depend on $m$ and the initial conditions. 
\end{lemma}
\begin{proof}

\begin{enumerate}
        \item Use Cauchy's mean-value theorem to observe that
\begin{equation*}
    \frac{\cos(2b(t))-\cos(2a(t))}{b^2(t)-a^2(t)} = -\frac{\sin(2\xi(t))}{\xi(t)} \simeq -1
\end{equation*}
with $a(t) < \xi < b(t)$. Thus \eqref{eqn:ode:simple} implies that
\begin{equation}
\label{ineq:b:diff}
   -b^2(t) + a^2(t)\lesssim b'(t) \lesssim -b^2(t) + a^2(t)
\end{equation}
In particular, $b'(t) \gtrsim -b^2(t)$ and hence
\begin{equation}
\label{ineq:b:lowerbnd}
    b(t) \gtrsim \frac{1}{1+t}.
\end{equation}
\item In view of \Cref{lem:G-monotonicity}, we have
\begin{equation}
\label{ineq:a:b:comparison}
   \frac{a'(t)}{\sin(2a(t))} =2\frac{G(t, a(t))}{\sin(2a(t))} \le  2\frac{G(t, b(t))}{\sin(2b(t))} =    \frac{b'(t)}{\sin(2b(t))}
\end{equation}
and hence
\begin{equation*}
    \frac{\mathrm{d}}{\mathrm{d}t}\log \frac{\tan(a(t))}{\tan(b(t))} \le 0,
\end{equation*}
which implies that $ \frac{\tan(a(t))}{\tan(b(t))}$ is decreasing and hence
\begin{equation*}
    \frac{1- \frac{\tan(a(t))}{\tan(b(t))}}{1+ \frac{\tan(a(t))}{\tan(b(t))}} = \frac{\sin(b(t)-a(t))}{\sin(b(t) + a(t))}
\end{equation*}
is increasing.
From \eqref{eqn:ode:simple} and trigonometric identities, we compute
\begin{equation*}
    \begin{aligned}
         \frac{\mathrm{d}}{\mathrm{d}t}\log \frac{\tan(a(t))}{\tan(b(t))} = &-\frac{c}{s_m}\{\cos(2({\textstyle\frac{\pi}{m}}-b(t))) - \cos(2({\textstyle\frac{\pi}{m}}-a(t)))\}\\
         &-\frac{c}{s_m}\frac{\sin(2({\textstyle\frac{\pi}{m}}-b(t)))}{\sin(2b(t))}\{\cos(2b(t)) - \cos(2a(t))\} \\
         &= -2c\frac{\sin^2(b(t)-a(t))}{\sin(2b(t))}\\
         &= -2c\frac{\sin^2(b(t) + a(t))}{\sin(2b(t))} \frac{\sin^2(b(t)-a(t))}{\sin^2(b(t) + a(t))}\\
         &\le -2c\frac{\sin^2(b(t) + a(t))}{\sin(2b(t))} \frac{\sin^2(b(0)-a(0))}{\sin^2(b(0) + a(0))} \\
         &\lesssim -b(t) \\
         &\lesssim -\frac{1}{1+t}
    \end{aligned}
\end{equation*}
where the last inequality is by \eqref{ineq:b:lowerbnd}. It follows, since $(1+t)^{-1}$ is not integrable, that
\begin{equation*}
    \lim_{t\to\infty} \frac{\tan(a(t))}{\tan(b(t))} = 0
\end{equation*}
In particular, using $\tan x \simeq x$, we also have
\begin{equation*}
    \lim_{t\to\infty} \frac{a(t)}{b(t)} = 0.
\end{equation*}
In fact, one sees from
\begin{equation*}
    \begin{aligned}
         \frac{\mathrm{d}}{\mathrm{d}t}\log \frac{\tan(a(t))}{\tan(b(t))} 
         &\lesssim -b(t) \\
         &\lesssim -\frac{1}{1+t},
    \end{aligned}
\end{equation*}
 $\tan x \simeq x$ and $b(t) \lesssim (1+t)^{-1}$ that
\begin{equation*}
    a(t) \lesssim (1+t)^{-1-\gamma}
\end{equation*}
for some $\gamma > 0$.
\item 
Returning to \eqref{ineq:b:diff}, we see that
\begin{equation*}
    b'(t) \lesssim -b^2(t) + o(b^2(t))
\end{equation*}
and hence
\begin{equation}
    b(t) \lesssim \frac{1}{1+t}.
\end{equation}

\end{enumerate}
\end{proof}

Note in particular that the associated function $g(t,\theta) = c\mathbf{1}_{(a(t), b(t)]}$ satisfies
\begin{equation*}
\|g(t)\|_{L^1} = c(b(t) - a(t)) \simeq \frac{1}{1+t}
\end{equation*}

\begin{remark}
\label[remark]{rmk:non-endpoint:decay}
    The estimates obtained in \Cref{lem:segment:quant} are representative of the behavior of general step functions as we will now show. See also Figure~\ref{fig:single-packet-snapshots}. Let us draw the reader's attention to the conclusion $a(t) \lesssim b(t)(1+t)^{-\gamma}$, which means that the left endpoint decays faster than the right one. This observation and its guises will play a role in closing certain non-linear estimates below.
\end{remark}  
\subsection{Step functions}
The proof of the following proposition will be similar to that of \Cref{lem:segment:quant}, but more general. 
\begin{proposition}
\label{prop:single:stepfn}
Assume that 
\begin{equation*}
    g_0(\theta) = \sum_{k=1}^n c_k\mathbf{1}_{(a_k, b_k]}(\theta)
\end{equation*}
where  $0 < a_1 < b_1 < a_2 < b_2 < \cdots < a_n < b_n < \frac{\pi}{m}$ and $0 < c_j \le 1$. The unique solution $g(t)$ to \eqref{eqn:Euler:reduced:transport}-\eqref{eqn:Euler:reduced:BSlaw} satisfies the following properties.
\begin{equation*}
    \begin{aligned}
        \frac{1}{1+t} &\lesssim \|g(t)\|_{L^1} \lesssim \frac{1}{1+t} \\
        \sup \operatorname{supp} g(t) &\lesssim \frac{1}{1+t} \\
        \int_0^{\frac{\pi}{m}} \sin(2\theta)g(t,\theta)\  \mathrm{d}\theta &\lesssim \frac{1}{(1+t)^2},
    \end{aligned}
\end{equation*}
where the implicit constants depend on $g_0$ and $m$. If $x(t)$ is the trajectory of any endpoint other than $b_n$, then
\begin{equation*}
    \int_0^\infty x(t)\  \mathrm{d}t < \infty.
\end{equation*}
\end{proposition}

\begin{proof}
\begin{enumerate}
\item 
Since $g_0$ is transported, it remains a step function for all time:
\begin{equation*}
   g(t,\theta) = \sum_{k=1}^n c_k\mathbf{1}_{(a_k(t), b_k(t)]}(\theta)
\end{equation*}
where 
\begin{equation*}
 \begin{aligned}
        a_k'(t) = 2G(t, a_k(t)), \quad a_k(0) = a_k \\
         b_k'(t) = 2G(t, b_k(t)), \quad b_k(0) = b_k \\
 \end{aligned}
\end{equation*} 
for $k = 1, \dots, n$.
Moreover, since characteristics do not cross, the topology of $g_0$ is preserved in time:
\begin{equation*}
    0 < a_1(t) < b_1(t) < a_2(t) < b_2(t) < \cdots < a_n(t) < b_n(t) < \frac{\pi}{m}
\end{equation*}
We can compute in detail the system of ODEs defining $a_k(t)$ and $b_k(t)$. We have

\begin{equation}
    \label{b_k:ode:explicit}
        \begin{aligned}
            \quad b_k'(t) &= 2G(t, b_k(t)) \\
            &= 2\int_{a_1(t)}^{b_n(t)} K_m(b_k(t), \zeta) g(t,\zeta) \ \mathrm{d}\zeta \\
            &= 2\sum_{j=1}^n c_j\int_{a_j(t)}^{b_j(t)} K_m(b_k(t), \zeta) \ \mathrm{d}\zeta \\
            &= 2\sum_{j=1}^{k} c_j\int_{a_j(t)}^{b_j(t)} K_m(b_k(t), \zeta) \ \mathrm{d}\zeta + 2\sum_{j=k+1}^n c_j\int_{a_j(t)}^{b_j(t)} K_m(b_k(t), \zeta) \ \mathrm{d}\zeta \\
            &= -\frac{1}{s_m}\sin(2({\textstyle\frac{\pi}{m}}-b_k(t)))\sum_{j=1}^{k} c_j\int_{a_j(t)}^{b_j(t)} \sin(2\zeta) \ \mathrm{d}\zeta \\
            & \ - \frac{1}{s_m}\sin(2b_k(t))\sum_{j=k+1}^{n} c_j\int_{a_j(t)}^{b_j(t)} \sin(2({\textstyle\frac{\pi}{m}}-\zeta)) \ \mathrm{d}\zeta\\
            &= \frac{1}{2s_m}c_k\sin(2({\textstyle\frac{\pi}{m}}-b_k(t)))\{\cos(2b_k(t)) - \cos(2a_k(t))\}\\
            &+ \frac{1}{2s_m}\sin(2({\textstyle\frac{\pi}{m}}-b_k(t)))\sum_{j=1}^{k-1} c_j\{\cos(2b_j(t)) - \cos(2a_j(t))\} \\
            & \ - \frac{1}{2s_m}\sin(2b_k(t))\sum_{j=k+1}^{n} c_j\left\{\cos(2({\textstyle\frac{\pi}{m}}-b_j(t))) - \cos(2({\textstyle\frac{\pi}{m}}-a_j(t)))\right\}\\
        \end{aligned}
    \end{equation}

    and likewise
    \begin{equation*}
        \begin{aligned}
            a_k'(t) &= -\frac{1}{2s_m}c_k\sin(2a_k(t))\left\{\cos(2({\textstyle\frac{\pi}{m}}-b_k(t))) - \cos(2({\textstyle\frac{\pi}{m}}-a_k(t)))\right\}\\
            &+ \frac{1}{2s_m}\sin(2({\textstyle\frac{\pi}{m}}-a_k(t)))\sum_{j=1}^{k-1} c_j\{\cos(2b_j(t)) - \cos(2a_j(t))\} \\
            & \ - \frac{1}{2s_m}\sin\left(2a_k(t)\right)\sum_{j=k+1}^{n} c_j\left\{\cos(2({\textstyle\frac{\pi}{m}}-b_j(t))) - \cos(2({\textstyle\frac{\pi}{m}}-a_j(t)))\right\}\\
        \end{aligned}
    \end{equation*}

\item In view of \Cref{lem:G-monotonicity}, we have
\begin{equation}
   \frac{a_n'(t)}{\sin(2a_n(t))} =2\frac{G(t, a_n(t))}{\sin(2a_n(t))} \le  2\frac{G(t, b_n(t))}{\sin(2b_n(t))} =    \frac{b_n'(t)}{\sin(2b_n(t))}
\end{equation}
and hence
\begin{equation*}
    \frac{\mathrm{d}}{\mathrm{d}t}\log \frac{\tan(a_n(t))}{\tan(b_n(t))} \le 0,
\end{equation*}
Note that we can drop negative terms in \eqref{b_k:ode:explicit} to get
\begin{equation*}
    \begin{aligned}
        b_n'(t) &\le \frac{c_n}{2s_m}\sin(2({\textstyle\frac{\pi}{m}}-b_n(t)))\{\cos(2b_n(t)) - \cos(2a_n(t))\} \\
        &= \frac{c_n}{s_m}\sin(2({\textstyle\frac{\pi}{m}}-b_n(t)))\{\sin^2(a_n(t)) - \sin^2(b_n(t))\} \\
        &= -\frac{c_n}{s_m}\sin(2({\textstyle\frac{\pi}{m}}-b_n(t)))\left\{1- \frac{\sin^2(a_n(t))}{\sin^2(b_n(t))}\right\} \sin^2(b_n(t))
    \end{aligned}
\end{equation*}
Since $ \frac{\mathrm{d}}{\mathrm{d}t}\log \frac{\tan(a_n(t))}{\tan(b_n(t))} \le 0$, $a_n'(t) \le 0$, $b_n'(t) \le 0$, and $a_n(t) \le b_n(t)$, the ratio $\frac{\sin(a_n(t))}{\sin(b_n(t))}$ is decreasing. Indeed, writing $A=(\log\tan a_n)'$ and $B=(\log\tan b_n)'$, we have $A-B\le0$, $B\le0$, and
\[
 (\log(\sin a_n/\sin b_n))'=\cos^2(a_n)(A-B)+(\cos^2(a_n)-\cos^2(b_n))B\le0.
\] So the last term above is bounded by (note $\frac{\pi}{m}-b_n(t) \gtrsim 1$)
\begin{equation*}
    -\frac{c_n}{s_m}\sin(2({\textstyle\frac{\pi}{m}}-b_n(t)))\left\{1- \frac{\sin^2(a_n(0))}{\sin^2(b_n(0))}\right\} \sin^2(b_n(t)) \lesssim -\sin^2(b_n(t)) \lesssim -b_n^2(t)
\end{equation*}
Therefore $b_n'(t) \lesssim -b_n^2(t)$ and we have $b_n(t) \lesssim \frac{1}{1+t}$.
Since $\sup \operatorname{supp} g(t) = b_n(t) \lesssim \frac{1}{1+t}$, the second estimate in the statement of the proposition is proved.
\item We now prove that
\begin{equation*}
    b_n(t) \gtrsim \frac{1}{1+t}
\end{equation*}
From \eqref{b_k:ode:explicit}, and $0 < c_j \le 1$, we have
\begin{equation*}
\begin{aligned}
        b_n'(t) &= \frac{1}{2s_m}\sin(2({\textstyle\frac{\pi}{m}}-b_n(t)))\sum_{k=1}^{n}c_k\{\cos(2b_k(t)) - \cos(2a_k(t))\} \\
        &\ge \frac{1}{2s_m}\sin(2({\textstyle\frac{\pi}{m}}-b_n(t)))\sum_{k=1}^{n}\{\cos(2b_k(t)) - \cos(2a_k(t))\} \\
\end{aligned}
\end{equation*}
Using 
\begin{equation*}
    \frac{\cos(2b_k(t))-\cos(2a_k(t))}{b_k^2(t)-a_k^2(t)}  \simeq -1, \quad \sin(2({\textstyle\frac{\pi}{m}}-b_n(t))) \gtrsim 1,
\end{equation*}
in the preceding estimate, we see that 
\begin{equation*}
    \begin{aligned}
        b_n'(t) 
        &\gtrsim \sum_{k=1}^{n}\{-b_k^2(t) +a_k^2(t)\} \\
        &\ge \sum_{k=1}^{n}-b_k^2(t) \ge -nb_n^2(t) \\
        &\gtrsim -b_n^2(t)
\end{aligned}
\end{equation*}
and consequently
\begin{equation*}
    b_n(t) \gtrsim \frac{1}{1+t}
\end{equation*}
\item 

Next, we prove that 
\begin{equation*}
\int_0^\infty a_n(t) \ \mathrm{d}t < \infty
\end{equation*}
Compute
\begin{equation*}
    \begin{aligned}
         \frac{\mathrm{d}}{\mathrm{d}t}\log \frac{\tan(a_n(t))}{\tan(b_n(t))} = &\frac{1}{s_m}\frac{\sin(2({\textstyle\frac{\pi}{m}}-a_n(t)))}{\sin(2a_n(t))}\sum_{k=1}^{n-1}c_k\{\cos(2b_k(t)) - \cos(2a_k(t))\}\\
         &-\frac{1}{s_m}\frac{\sin(2({\textstyle\frac{\pi}{m}}-b_n(t)))}{\sin(2b_n(t))}\sum_{k=1}^{n-1}c_k\{\cos(2b_k(t)) - \cos(2a_k(t))\} \\
         &-\frac{c_n}{s_m}\{\cos(2({\textstyle\frac{\pi}{m}}-b_n(t))) - \cos(2({\textstyle\frac{\pi}{m}}-a_n(t)))\}\\
         &-\frac{c_n}{s_m}\frac{\sin(2({\textstyle\frac{\pi}{m}}-b_n(t)))}{\sin(2b_n(t))}\{\cos(2b_n(t)) - \cos(2a_n(t))\} \\
         &\le -\frac{1}{s_m}\frac{\sin(2({\textstyle\frac{\pi}{m}}-b_n(t)))}{\sin(2b_n(t))}\sum_{k=1}^{n-1}c_k\{\cos(2b_k(t)) - \cos(2a_k(t))\} \\
         &-\frac{c_n}{s_m}\{\cos(2({\textstyle\frac{\pi}{m}}-b_n(t))) - \cos(2({\textstyle\frac{\pi}{m}}-a_n(t)))\}\\
         &-\frac{c_n}{s_m}\frac{\sin(2({\textstyle\frac{\pi}{m}}-b_n(t)))}{\sin(2b_n(t))}\{\cos(2b_n(t)) - \cos(2a_n(t))\} \\
    \end{aligned}
\end{equation*}
The last two terms above are handled as in \Cref{lem:segment:quant}; their sum is bounded by
\begin{equation*}
    -2c_n\frac{\sin^2(b_n(t) + a_n(t))}{\sin(2b_n(t))} \frac{\sin^2(b_n(0)-a_n(0))}{\sin^2(b_n(0) + a_n(0))} \lesssim -b_n(t)
\end{equation*}
The first two terms can be written as
\begin{equation*}
    \begin{aligned}
        \frac{1}{s_m}\left(\frac{\sin(2({\textstyle\frac{\pi}{m}}-a_n(t)))}{\sin(2a_n(t))}-\frac{\sin(2({\textstyle\frac{\pi}{m}}-b_n(t)))}{\sin(2b_n(t))}\right)\sum_{k=1}^{n-1}c_k\{\cos(2b_k(t)) - \cos(2a_k(t))\}
    \end{aligned}
\end{equation*}
Since $a_k < b_k$, the second factor is $ < 0$. Since 
\begin{equation*}
    \frac{\sin(2({\textstyle\frac{\pi}{m}}-\theta))}{\sin(2\theta)}
\end{equation*}
is decreasing on $(0, \pi/m)$, the first factor is $ > 0$. Inserting this above, we have
\begin{equation*}
      \frac{\mathrm{d}}{\mathrm{d}t}\log \frac{\tan(a_n(t))}{\tan(b_n(t))} \lesssim -b_n(t)
\end{equation*}
and so
\begin{equation*}
\frac{\mathrm{d}}{\mathrm{d}t} \frac{\tan(a_n(t))}{\tan(b_n(t))} \lesssim -b_n(t)\frac{\tan(a_n(t))}{\tan(b_n(t))} \lesssim -a_n(t)
\end{equation*}
since $\tan(x) \simeq x$ for $x \in [0, \pi/m]$.
Therefore
\begin{equation*}
    a_n(t) \lesssim - \frac{\mathrm{d}}{\mathrm{d}t} \frac{\tan(a_n(t))}{\tan(b_n(t))}
\end{equation*}
and integrating,
\begin{equation*}
    \int_0^t a_n(s) \ \mathrm{d}s \lesssim \frac{\tan(a_n(0))}{\tan(b_n(0))}- \frac{\tan(a_n(t))}{\tan(b_n(t))} < \frac{\tan(a_n(0))}{\tan(b_n(0))} < 1.
\end{equation*}
Hence
\begin{equation*}
\int_0^\infty a_n(t) \ \mathrm{d}t < \infty
\end{equation*}
It is immediate that all endpoints to the left of the right-most one go to zero at an integrable rate. 
\item Note that
\begin{equation*}
\begin{aligned}
        \|g(t)\|_{L^1} &= \sum_{k=1}^n c_k (b_k(t) - a_k(t))\\&= b_n(t)\left(c_n-c_n\frac{a_n(t)}{b_n(t)}+\sum_{k=1}^{n-1} c_k\frac{(b_k(t) - a_k(t))}{b_n(t)}\right)
\end{aligned}
\end{equation*}

 since
\begin{equation*}
    \left|\sum_{k=1}^{n-1} c_k\frac{(b_k(t) - a_k(t))}{b_n(t)}\right| \le \sum_{k=1}^{n-1} c_k\frac{a_n(t) }{b_n(t)} \lesssim 1, 
\end{equation*}
the term in parenthesis is $\lesssim 1$, the implicit constant depending on $n$ and the initial $a_k, b_k$. Thus
\begin{equation*}
    \|g(t)\|_{L^1} \lesssim b_n(t) \lesssim \frac{1}{1+t}
\end{equation*}
On the other hand, 
\begin{equation*}
\begin{aligned}
        \|g(t)\|_{L^1} &= \sum_{k=1}^n c_k (b_k(t) - a_k(t))\\
        &\ge c_n(b_n(t) - a_n(t))\\
        &= c_n\frac{\sin(b_n(t)) - \sin(a_n(t))}{\cos(\xi_n(t))} \\
        &\ge c_n \sin(b_n(t))\left(1-\frac{\sin(a_n(t))}{\sin(b_n(t))}\right) \\
        &\ge c_n\sin(b_n(t))\left(1-\frac{\sin(a_n(0))}{\sin(b_n(0))}\right)\\
        &\gtrsim \sin(b_n(t)) \gtrsim b_n(t) \gtrsim \frac{1}{1+t}
\end{aligned}
\end{equation*}
Lastly, 
\begin{equation*}
\begin{aligned}
        \int_0^{\frac{\pi}{m}}\sin(2\theta)g(t,\theta)  \ \mathrm{d}\theta &=  \int_{0}^{b_n(t)}\sin(2\theta)g(t,\theta)  \ \mathrm{d}\theta \\
        &\le \sin(2b_n(t))\|g(t)\|_{L^1} \\
        &\le 2b_n(t)\|g(t)\|_{L^1}  \\
        &\lesssim \frac{1}{(1+t)^2}
\end{aligned}
\end{equation*}

\end{enumerate}
\end{proof}

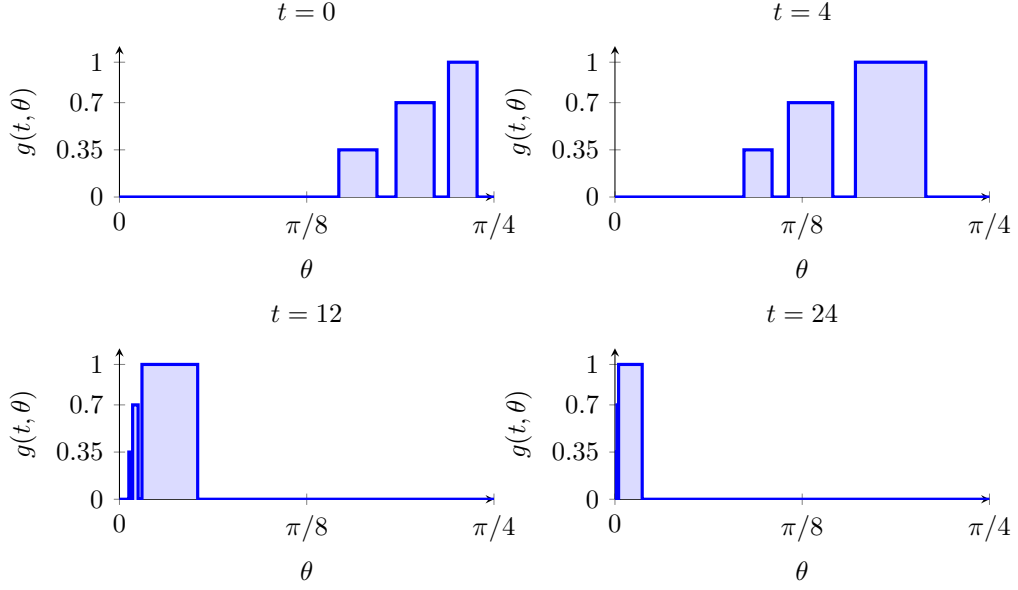
\begin{figure}[t]
\centering
\begin{tikzpicture}
\begin{groupplot}[
  group style={group size=2 by 2, horizontal sep=1.6cm, vertical sep=2.0cm},
  width=0.42\textwidth,
  height=0.23\textwidth,
  xmin=0, xmax=0.785398,
  ymin=0, ymax=1.12,
  xtick={0,0.392699,0.785398},
  xticklabels={$0$,$\pi/8$,$\pi/4$},
  ytick={0,0.35,0.7,1},
  xlabel={$\theta$}, ylabel={$g(t,\theta)$},
  axis lines=left,
  tick label style={font=\small},
  label style={font=\small},
  title style={font=\small},
]
\nextgroupplot[title={$t=0$}]
\addplot[draw=none, fill=blue!14] coordinates {(0.460000,0) (0.460000,0.350000) (0.540000,0.350000) (0.540000,0)} -- cycle;
\addplot[draw=none, fill=blue!14] coordinates {(0.580000,0) (0.580000,0.700000) (0.660000,0.700000) (0.660000,0)} -- cycle;
\addplot[draw=none, fill=blue!14] coordinates {(0.690000,0) (0.690000,1.000000) (0.750000,1.000000) (0.750000,0)} -- cycle;
\addplot[blue, very thick] coordinates {(0.000000,0.000000) (0.460000,0.000000) (0.460000,0.350000) (0.540000,0.350000) (0.540000,0.000000) (0.580000,0.000000) (0.580000,0.000000) (0.580000,0.700000) (0.660000,0.700000) (0.660000,0.000000) (0.690000,0.000000) (0.690000,0.000000) (0.690000,1.000000) (0.750000,1.000000) (0.750000,0.000000) (0.785398,0.000000)};
\nextgroupplot[title={$t=4$}]
\addplot[draw=none, fill=blue!14] coordinates {(0.270536,0) (0.270536,0.350000) (0.329516,0.350000) (0.329516,0)} -- cycle;
\addplot[draw=none, fill=blue!14] coordinates {(0.363689,0) (0.363689,0.700000) (0.456797,0.700000) (0.456797,0)} -- cycle;
\addplot[draw=none, fill=blue!14] coordinates {(0.504316,0) (0.504316,1.000000) (0.651808,1.000000) (0.651808,0)} -- cycle;
\addplot[blue, very thick] coordinates {(0.000000,0.000000) (0.270536,0.000000) (0.270536,0.350000) (0.329516,0.350000) (0.329516,0.000000) (0.363689,0.000000) (0.363689,0.000000) (0.363689,0.700000) (0.456797,0.700000) (0.456797,0.000000) (0.504316,0.000000) (0.504316,0.000000) (0.504316,1.000000) (0.651808,1.000000) (0.651808,0.000000) (0.785398,0.000000)};
\nextgroupplot[title={$t=12$}]
\addplot[draw=none, fill=blue!14] coordinates {(0.019644,0) (0.019644,0.350000) (0.024523,0.350000) (0.024523,0)} -- cycle;
\addplot[draw=none, fill=blue!14] coordinates {(0.027632,0) (0.027632,0.700000) (0.039103,0.700000) (0.039103,0)} -- cycle;
\addplot[draw=none, fill=blue!14] coordinates {(0.047189,0) (0.047189,1.000000) (0.164021,1.000000) (0.164021,0)} -- cycle;
\addplot[blue, very thick] coordinates {(0.000000,0.000000) (0.019644,0.000000) (0.019644,0.350000) (0.024523,0.350000) (0.024523,0.000000) (0.027632,0.000000) (0.027632,0.000000) (0.027632,0.700000) (0.039103,0.700000) (0.039103,0.000000) (0.047189,0.000000) (0.047189,0.000000) (0.047189,1.000000) (0.164021,1.000000) (0.164021,0.000000) (0.785398,0.000000)};
\nextgroupplot[title={$t=24$}]
\addplot[draw=none, fill=blue!14] coordinates {(0.003063,0) (0.003063,0.350000) (0.003830,0.350000) (0.003830,0)} -- cycle;
\addplot[draw=none, fill=blue!14] coordinates {(0.004323,0) (0.004323,0.700000) (0.006215,0.700000) (0.006215,0)} -- cycle;
\addplot[draw=none, fill=blue!14] coordinates {(0.007604,0) (0.007604,1.000000) (0.057123,1.000000) (0.057123,0)} -- cycle;
\addplot[blue, very thick] coordinates {(0.000000,0.000000) (0.003063,0.000000) (0.003063,0.350000) (0.003830,0.350000) (0.003830,0.000000) (0.004323,0.000000) (0.004323,0.000000) (0.004323,0.700000) (0.006215,0.700000) (0.006215,0.000000) (0.007604,0.000000) (0.007604,0.000000) (0.007604,1.000000) (0.057123,1.000000) (0.057123,0.000000) (0.785398,0.000000)};
\end{groupplot}
\end{tikzpicture}
\caption{Evolution of a step function for \eqref{eqn:Euler:S1:transport}-\eqref{eqn:Euler:S1:BSlaw} with $m=4$.}
\label{fig:single-packet-snapshots}
\end{figure}

We record in the following some corollaries for use in the sequel. Consider the easy problem of designing an initial configuration 
\begin{equation*}
    g_0(\theta) = \sum_{k=1}^n c_k\mathbf{1}_{(a_{k}^0, b_k^0]}(\theta)
\end{equation*}
to achieve a desired target step function $h(\theta)$ at time $T$:
\begin{equation*}
    g(T, \theta) = h(\theta) := \sum_{k=1}^n c_k\mathbf{1}_{(a_k, b_k]}(\theta)
\end{equation*}
This is a simple matter of solving our system backwards in time beginning at time $T$. Make the change of variables $\zeta =\frac{\pi}{m}-\theta$ and $\tau = T - t$, $g(t,\theta) = f(\tau, \zeta)$, $G(t, \theta) = F(\tau, \zeta)$. Then we have
\begin{equation}
\label{eqn:sys:reflected}
\begin{aligned}
        \partial_\tau f + 2F\partial_\zeta f &= 0 \\
        \partial_{\zeta\zeta}F + 4F &= f \\
        f(0, \zeta) &= h({\textstyle\frac{\pi}{m}}-\zeta)
\end{aligned}
\end{equation}
The reflected terminal datum is again an ordered step function with the same coefficients; after relabeling the reflected endpoints, \Cref{prop:single:stepfn} applies. Undoing the change of variables gives the following corollary. 

\begin{corollary}
Let
\begin{equation*}
    h(\theta) = \sum_{k=1}^n c_k\mathbf{1}_{(a_k, b_k]}(\theta)
\end{equation*}
where  $0 < a_1 < b_1 < a_2 < b_2 < \cdots < a_n < b_n < \frac{\pi}{m}$ and $0 < c_j \le 1$. Let $g(t)$ be the unique solution to \eqref{eqn:Euler:reduced:transport}-\eqref{eqn:Euler:reduced:BSlaw} which satisfies $g(T) = h$. Then for $t \le T$, the following estimates hold.  
\begin{equation*}
    \begin{aligned}
        \frac{1}{1+T-t} &\lesssim \|g(t)\|_{L^1} \lesssim \frac{1}{1+T-t} \\
        \frac{\pi}{m}-\inf \operatorname{supp} g(t) &\lesssim \frac{1}{1+T-t} \\
        \int_0^{\frac{\pi}{m}} \sin(2({\textstyle\frac{\pi}{m}}-\theta))g(t,\theta)\  \mathrm{d}\theta &\lesssim \frac{1}{(1+T-t)^2},
\end{aligned}
\end{equation*} 
with the implied constants independent of $T$.
\end{corollary}

\begin{corollary}
\label[corollary]{lem:single:stepfn:perturbed:fwd}
Let $g(t,\theta)$ satisfy
    \begin{equation*}
    \begin{aligned}
        \partial_t g + (2G - \mu(t)\sin(2\theta)) \partial_\theta g &= 0 \\
                G(t, \theta)&= \int_0^{\frac{\pi}{m}}K_m(\theta, \zeta)g(t,\zeta) \ \mathrm{d}\zeta \\
                g(0, \theta) &= g_0(\theta),
 \end{aligned}
\end{equation*}
where $g_0(\theta)$ is as defined in \Cref{prop:single:stepfn} above and $\mu \in C^0([0,\infty))$ with $\mu(t) \ge 0$. Write
\[
 g(t,\theta)=\sum_{k=1}^n c_k\mathbf 1_{(a_k(t),b_k(t)]}(\theta),
 \qquad 0<a_1(t)<b_1(t)<\cdots<a_n(t)<b_n(t)<\frac{\pi}{m}.
\]
Then, for $t\ge 0$, we have
\begin{equation*}
    \begin{aligned}
       \|g(t)\|_{L^1} \lesssim \frac{1}{1+t} \\
        \sup \operatorname{supp} g(t) &\lesssim \frac{1}{1+t} \\
        \int_0^{\frac{\pi}{m}} \sin(2\theta)g(t,\theta)\  \mathrm{d}\theta &\lesssim \frac{1}{(1+t)^2}.
\end{aligned}
\end{equation*} 
If $x(t)$ is the trajectory of any endpoint other than the right-most one $b_n(t)=\sup \operatorname{supp}g(t)$, then
\begin{equation*}
    \int_0^\infty x(t)\  \mathrm{d}t < \infty.
\end{equation*}
If, in addition, $\mu\in L^1(0,\infty)$, then there is $\gamma>0$ such that
\begin{equation}
\label{eq:perturbed-forward-refined}
 b_n(t)\simeq \frac1{1+t},\qquad
 \frac{a_n(t)}{b_n(t)}\lesssim (1+t)^{-\gamma},
 \qquad a_n(t)\lesssim (1+t)^{-1-\gamma}.
\end{equation}
Moreover, after decreasing $\gamma$ if necessary,
\begin{equation}
\label{eq:perturbed-forward-sharp}
 c_n b_n(t)\le \frac1{1+t}+\frac{C}{(1+t)^{1+\gamma}}.
\end{equation}
The constants in the last two estimates depend only on $g_0,m$, and an upper bound for $\|\mu\|_{L^1}$.
\end{corollary}

\begin{proof}
Each endpoint $z(t)$ satisfies
\begin{equation}
\label{eq:perturbed-endpoint}
 z'(t)=2G(t,z(t))-\mu(t)\sin(2z(t)).
\end{equation}
The additional term is non-positive.  The proof of the upper bound for $b_n$ in
\Cref{prop:single:stepfn} therefore applies without change and gives
$b_n(t)\lesssim(1+t)^{-1}$.  In particular, the asserted bounds for $\|g(t)\|_{L^1}$ and 
\[ \int_0^\frac{\pi}{m} \sin(2\theta) g(t,\theta)  \ \mathrm{d}\theta\]
follow exactly as in that proposition.

For completeness, put $r(t)=\tan a_n(t)/\tan b_n(t)$.  Since
\[
 \frac{\mathrm d}{\mathrm dt}\log\tan z(t)=\frac{2z'(t)}{\sin(2z(t))},
\]
the perturbation in \eqref{eq:perturbed-endpoint} contributes the same term
$-2\mu(t)$ at $a_n$ and at $b_n$, and hence cancels from $(\log r)'$.  The
calculation in the proof of \Cref{prop:single:stepfn} consequently gives
\begin{equation}
\label{eq:perturbed-ratio-differential}
 (\log r)'(t)\le-\kappa b_n(t),
 \qquad r'(t)\le-\kappa_1 a_n(t),
\end{equation}
for constants $\kappa,\kappa_1>0$ depending only on the initial step function
and on $m$.  The second inequality yields $\int_0^\infty a_n(t)\,\mathrm dt<\infty$;
all the remaining endpoints lie to the left of $a_n$, so they are integrable as
well.

Assume now that $\mu\in L^1$.  The explicit endpoint equation and the ordering
give
\[
 b_n'(t)\ge-Cb_n(t)^2-C\mu(t)b_n(t).
\]
Thus $z=1/b_n$ satisfies $z'\le C+C\mu z$, and Gr\"{o}nwall's inequality gives
$b_n(t)\gtrsim(1+t)^{-1}$, with a constant depending only on the quantities
specified in the statement.  Integrating the first inequality in
\eqref{eq:perturbed-ratio-differential} now gives
$r(t)\lesssim(1+t)^{-\gamma}$ for some $\gamma>0$.  Since $\tan z\simeq z$ on
$[0,\pi/m]$, this proves \eqref{eq:perturbed-forward-refined}.

It remains to prove \eqref{eq:perturbed-forward-sharp}. Dropping the non-positive terms in the $b_n$ ODE yields
\[
 b_n'\le-\frac{c_n}{s_m}\sin\!\left(2\Big(\frac\pi m-b_n\Big)\right)
 \left(1-\frac{\sin^2a_n}{\sin^2b_n}\right)\sin^2b_n.
\]
The estimates already proved imply
\[
 b_n(t)\lesssim(1+t)^{-1},\qquad
 \frac{\sin a_n(t)}{\sin b_n(t)}\lesssim(1+t)^{-\gamma}.
\]
Consequently,
\[
 \frac{\sin(2({\textstyle\frac{\pi}{m}}-b_n))}{s_m}
     =1+O((1+t)^{-1}),\qquad
 \frac{\sin^2b_n}{b_n^2}=1+O((1+t)^{-2}),
\]
and
\[
 1-\frac{\sin^2a_n}{\sin^2b_n}
     =1+O((1+t)^{-2\gamma}).
\]
Thus, with $\gamma_0=\min\{1,2\gamma\}>0$,
\begin{equation*}
 b_n'(t)\le-c_n\bigl(1-C(1+t)^{-\gamma_0}\bigr)b_n(t)^2.
\end{equation*}
Equivalently,
\[
 \frac{\mathrm d}{\mathrm dt}\frac1{b_n(t)}
 \ge c_n\bigl(1-C(1+t)^{-\gamma_0}\bigr).
\]
After decreasing $\gamma_0$ slightly if necessary, integration from a fixed
time $t_0$ to $t$ gives
\[
 \frac1{b_n(t)}\ge c_nt-C(1+t)^{1-\gamma_0}-C.
\]
For all sufficiently large $t$, the right-hand side is positive and can be
written as
\[
 c_n(1+t)\bigl(1-O((1+t)^{-\gamma_1})\bigr)
\]
for some $0<\gamma_1\le\gamma_0$. Inverting and using
$(1-z)^{-1}\le1+2z$ for $0\le z\le1/2$ yields
\[
 c_nb_n(t)\le \frac1{1+t}+\frac{C}{(1+t)^{1+\gamma_1}}.
\]
Increasing $C$ handles the bounded interval of earlier times. Renaming
$\gamma_1$ as $\gamma$ proves \eqref{eq:perturbed-forward-sharp}.
\end{proof}

By employing the change of variables $\zeta =\frac{\pi}{m}-\theta$ and $\tau = T - t$, above we obtain the corresponding conclusion for the backwards problem:
\begin{corollary}
\label[corollary]{prop:single:stepfn:perturbed}
Let
\[
 h(\theta)=\sum_{k=1}^n c_k\mathbf 1_{(a_k,b_k]}(\theta),
 \qquad 0<a_1<b_1<\cdots<a_n<b_n<\frac{\pi}{m},
 \qquad 0<c_k\le1,
\]
and let  $\mu \in C^0([0,\infty))$ with $\mu(t) \ge 0$. Let $g(t,\theta)$ satisfy
    \begin{equation*}
    \begin{aligned}
        \partial_t g + \bigl(2G - \mu(t)\sin(2({\textstyle\frac{\pi}{m}}-\theta))\bigr) \partial_\theta g &= 0 \\
                G(t, \theta)&= \int_0^{\frac{\pi}{m}}K_m(\theta, \zeta)g(t,\zeta) \ \mathrm{d}\zeta \\
                g(T, \theta) &= h.
 \end{aligned}
\end{equation*}
Then, for $0 \le t \le T$, we have
\begin{equation*}
    \begin{aligned}
   \|g(t)\|_{L^1} \lesssim \frac{1}{1+T-t} \\
        \frac{\pi}{m}-\inf \operatorname{supp} g(t) &\lesssim \frac{1}{1+T-t} \\
        \int_0^{\frac{\pi}{m}} \sin(2({\textstyle\frac{\pi}{m}}-\theta))g(t,\theta)\  \mathrm{d}\theta &\lesssim \frac{1}{(1+T-t)^2}.
\end{aligned}
\end{equation*}
and the implicit constants are independent of $T$.  Suppose further that
\begin{equation}
\label{eq:mu-uniform-integrability}
 \int_0^T\mu(t)\,\mathrm dt\le C_\mu
\end{equation}
with $C_\mu$ independent of $T$. After
reversing the labels and the coefficients, write the reflected endpoints as
\begin{equation}
\label{eq:reflected-endpoints-corollary}
 x_k(t)={\textstyle\frac{\pi}{m}}-b_{n+1-k}(t),\qquad
 y_k(t)={\textstyle\frac{\pi}{m}}-a_{n+1-k}(t),
 \qquad 0<x_1<y_1<\cdots<x_n<y_n.
\end{equation}
Then there are constants $C,\gamma>0$, independent of $T$, such that
\begin{equation}
\label{eq:backward-refined-endpoints}
 y_n(t)\simeq (1+T-t)^{-1},\qquad
 x_n(t)\le C(1+T-t)^{-1-\gamma},
\end{equation}
and
\begin{equation}
\label{eq:backward-sharp-endpoint}
 c_n y_n(t)\le (1+T-t)^{-1}+C(1+T-t)^{-1-\gamma}.
\end{equation}
In particular,
\begin{equation}
\label{eq:backward-integrable-remainder}
 \int_0^T\bigl(x_n(t)+y_n(t)^2\bigr)\,\mathrm dt\le C.
\end{equation}
\end{corollary}

\begin{proof}
Set $\zeta={\textstyle\frac{\pi}{m}}-\theta$, $\tau=T-t$, and
$f(\tau,\zeta)=g(T-\tau,{\textstyle\frac{\pi}{m}}-\zeta)$. The symmetry of $K_m$ under simultaneous
reflection shows that $f$ satisfies the forward equation in
\Cref{lem:single:stepfn:perturbed:fwd} with
$\widetilde\mu(\tau)=\mu(T-\tau)$.  The first three estimates therefore follow
from the first part of that corollary.  Under
\eqref{eq:mu-uniform-integrability},
$\|\widetilde\mu\|_{L^1(0,T)}\le C_\mu$.
Applying \eqref{eq:perturbed-forward-refined}--\eqref{eq:perturbed-forward-sharp} on $[0, T]$
gives \eqref{eq:backward-refined-endpoints} and
\eqref{eq:backward-sharp-endpoint}.  Finally,
\eqref{eq:backward-integrable-remainder} follows by integrating
$(1+T-t)^{-1-\gamma}+(1+T-t)^{-2}$.
\end{proof}

\section{An auxiliary construction}
In this section, we consider the following auxiliary question. Suppose we are given two functions
\begin{equation*}
    h_1 := \sum_{k=1}^{n_1} c_{1,k}\mathbf{1}_{(a_{1,k}, b_{1,k}]}(\theta) \quad \text{and}  \quad h_2:=\sum_{k=1}^{n_2} c_{2,k}\mathbf{1}_{(a_{2,k}, b_{2,k}]}(\theta)
\end{equation*}
supported in $(0, \frac{\pi}{m})$ and the coefficients satisfy $0<c_{ik} \le 1$.
\begin{question}
\label[question]{auxiliary-question}
    Can we design an initial configuration $g_0$, a step function, so that the solution $g(t)$ to \eqref{eqn:Euler:reduced:transport}-\eqref{eqn:Euler:reduced:BSlaw} has the following properties?
\begin{enumerate}
    \item We can split $g$ as a sum of two step functions \begin{equation*}
        g(t,\theta) = g_1(t,\theta) + g_2(t, \theta)
         \end{equation*}
         with $ \operatorname{supp} g_1(t)$ lying to the left of $\operatorname{supp} g_2(t)$ and in particular $g_1, g_2$ have disjoint supports for all time. 
    \item For any $\varepsilon > 0$, there is some $T > 0$ such that
        \begin{equation*}
            g_1(0, \theta) = h_1, \quad \|g_2(0)\|_{L^1} \le \varepsilon
        \end{equation*}
        and 
        \begin{equation*}
            \|g_2(T, \theta) - h_2\|_{L^1} \le \varepsilon, \quad \|g_1(T)\|_{L^1} \le \varepsilon
        \end{equation*}
\end{enumerate}
\end{question}

\subsection{A two-target construction} 
Assume that there is a solution $g$ to \Cref{auxiliary-question}. We may then write 
\begin{equation*}
  \begin{aligned}
        g_1(t, \theta) &= \sum_{j=1}^{n_1} c_{1,j}\mathbf{1}_{(a_{1,j}(t), b_{1,j}(t)]}\\
        g_2(t, \theta) &= \sum_{j=1}^{n_2} c_{2,j}\mathbf{1}_{(a_{2,j}(t), b_{2,j}(t)]}\\        
  \end{aligned}
\end{equation*}
We now reduce the PDE on $g$ to a system of ODEs for the endpoint trajectories $a_{1,j}(t), b_{1,j}(t)$, $1\le j \le n_1$ of $g_1$ and $a_{2,j}(t), b_{2,j}(t)$, $1\le j \le n_2$ of $g_2$.
First split the velocity as follows.
\begin{equation*}
    G(t,\theta) = G_1(t,\theta) + G_2(t,\theta), 
\end{equation*}
where
\begin{equation}
\label{eqn:Gk:kernel:expr}
        G_k(t, \theta) := \int_0^{\frac{\pi}{m}}K_m(\theta, \zeta)g_k(t,\zeta) \ \mathrm{d}\zeta,
\end{equation}
for $k = 1,2$. Define also for a function $g$, the weighted integrals
\begin{equation}
    \begin{aligned}
        M[g](t) &:= \int_0^{\frac{\pi}{m}} g(t,\theta) \sin(2({\textstyle\frac{\pi}{m}}-\theta)) \ \mathrm{d} \theta,\\
         m[g](t) &:= \int_0^{\frac{\pi}{m}} g(t,\theta) \sin(2\theta) \ \mathrm{d} \theta
    \end{aligned}
\end{equation}
It is easily checked, in view of \eqref{eqn:G:kernel:expr} that for $\theta \in \mathrm{supp} \  g_1$, we have
\begin{equation}
\begin{aligned}    G(t, \theta) &= G_1(t, \theta) - \frac{1}{2s_m}\sin(2\theta) M[g_{2}](t)
\end{aligned}
\end{equation}
while for $\theta \in \mathrm{supp} \  g_2$, 

\begin{equation}
\begin{aligned}    
G(t, \theta) &= G_2(t, \theta) -\frac{1}{2s_m}\sin(2({\textstyle\frac{\pi}{m}}-\theta))m[g_{1}](t)
\end{aligned}
\end{equation}
Therefore, the trajectories of $g_1$ satisfy
\begin{equation}
\label{eqn:left-packet:ODE}
    \begin{aligned}
        a_{1,j}'(t) &= 2G_1(t, a_{1,j}(t))  -\frac{1}{s_m}\sin(2a_{1,j}(t)) M[g_2](t)\\
        b_{1,j}'(t) &= 2G_1(t, b_{1,j}(t))  -\frac{1}{s_m}\sin(2b_{1,j}(t)) M[g_2](t)\\
    \end{aligned}
\end{equation}
for $j=1,\dots, n_1$, with initial conditions decided by the endpoints in $h_1$, that is, $a_{1,k}(0)=a_{1,k}, b_{1,k}(0) = b_{1,k}$.
The trajectories of $g_2$ satisfy, on the other hand,
\begin{equation}
\label{eqn:right-packet:ODE}
    \begin{aligned}
        a_{2,j}'(t) &= 2G_2(t, a_{2,j}(t))  -\frac{1}{s_m}\sin(2({\textstyle\frac{\pi}{m}}-a_{2,j}(t))) m[g_1](t)\\
        b_{2,j}'(t) &= 2G_2(t, b_{2,j}(t))  -\frac{1}{s_m}\sin(2({\textstyle\frac{\pi}{m}}-b_{2,j}(t))) m[g_1](t)\\
    \end{aligned}
\end{equation}
for $j=1,\dots, n_2$, with terminal conditions decided by the endpoints in $h_2$, that is, $a_{2,k}(T)=a_{2,k}, b_{2,k}(T) = b_{2,k}$.

If $g_1(t,\theta)$ were evolving alone (i.e. if it solves \eqref{eqn:Euler:reduced:transport}-\eqref{eqn:Euler:reduced:BSlaw}), then a trajectory $x(t)$ of $g_1$ would simply satisfy $x'(t) = 2G_1(t, x(t))$. Thus the effect of $g_2$ on the trajectories of $g_1$ is the perturbation $-\frac{1}{s_m}\sin(2x(t)) M[g_2](t)$.  In other words,
\begin{equation*}
    \partial_t g_1 + (2G_1-s_m^{-1}\sin(2\theta) M[g_2](t)) \partial_\theta g_1 = 0
\end{equation*}
on $\operatorname{supp} g_1$. By the same token, the ODEs for the trajectories $x(t)$ of $g_2$ are perturbed by $-\frac{1}{s_m}\sin(2({\textstyle\frac{\pi}{m}}-x(t)))m[g_1](t)$ and 
\begin{equation*}
    \partial_t g_2 + (2G_2-s_m^{-1}\sin(2({\textstyle\frac{\pi}{m}}-\theta))m[g_1](t)) \partial_\theta g_2 = 0
\end{equation*}

We shall construct a solution \Cref{auxiliary-question} in the following way. Let $\bar{g}_1$ be the solution to 
\begin{equation*}
 \begin{aligned}
        \partial_t \bar{g}_1 + 2\bar{G}_1 \partial_\theta \bar{g}_1 &= 0 \\
                \bar{G}_1(t, \theta) &= \int_0^{\frac{\pi}{m}}K_m(\theta, \zeta)\bar{g}_1(t,\zeta) \ \mathrm{d}\zeta, \\
                \bar{g}_1(0, \theta) &= h_1
 \end{aligned}
\end{equation*}
and let $\bar{g}_2$ be the solution of
\begin{equation*}
 \begin{aligned}
        \partial_t \bar{g}_2 + (2\bar{G}_2 - s_m^{-1}\sin(2({\textstyle\frac{\pi}{m}}-\theta))m[\bar{g}_1](t))\partial_\theta \bar{g}_2 &= 0 \\
                \bar{G}_2(t, \theta) &= \int_0^{\frac{\pi}{m}}K_m(\theta, \zeta)\bar{g}_2(t,\zeta) \ \mathrm{d}\zeta, \\
                \bar{g}_2(T, \theta) &= h_2
 \end{aligned}
\end{equation*}

We will select our initial configuration as
\begin{equation}
\label{def:two-target:data}
    g_0(\theta) = h_1 + \bar{g}_2 (0, \theta)
\end{equation}

\begin{remark}
    This is in lieu of the perhaps natural first choice $h_1 + \tilde{g}(0)$ where $\tilde{g}(t)$ solves \eqref{eqn:Euler:reduced:transport}-\eqref{eqn:Euler:reduced:BSlaw} with $\tilde{g}(T) = h_2$.
\end{remark}

Now  by \Cref{prop:single:stepfn}, 
\begin{equation*}
    m[\bar{g}_1](t) \lesssim \frac{1}{(1+t)^2}, \quad t \ge 0.
\end{equation*}
Furthermore, \Cref{prop:single:stepfn:perturbed} implies that for $t \le T$,
\begin{equation}
\label{bar-g_2:bounds}
   \|\bar{g}_2(t)\|_{L^1} \lesssim \frac{1}{(1+T-t)}, \quad \frac{\pi}{m}-\inf \operatorname{supp} \bar{g}_2(t) \lesssim \frac{1}{1+T-t},  \quad M[\bar{g}_2](t) \lesssim \frac{1}{(1+T-t)^2}
\end{equation}
and in particular,
\begin{equation}
\label{g2:init}
    \|\bar{g}_2(0)\|_{L^1} \lesssim \frac{1}{(1+T)}
\end{equation}
and 
\begin{equation}
    \operatorname{supp} \bar{g}_2(0) \subset \left[\frac{\pi}{m}-\frac{C}{1+T}, \frac{\pi}{m}\right]
\end{equation}
Thus if $T$ is sufficiently large, then the supports of $h_1$ and $\bar{g}_2 (0, \theta)$ are disjoint.

Let $g$ be the solution to \eqref{eqn:Euler:reduced:transport}-\eqref{eqn:Euler:reduced:BSlaw} with initial configuration \eqref{def:two-target:data}. Since characteristics do not cross, $g$ is, for all time, a sum of two step functions of disjoint support:
\begin{equation}
    g(t, \theta) = g_1(t, \theta) + g_2(t, \theta)
\end{equation}
By the discussion above,
\begin{equation*}
    \partial_t g_1 + (2G_1-s_m^{-1}\sin(2\theta) M[g_2](t)) \partial_\theta g_1 = 0
\end{equation*}
$g_1(0, \theta) = h_1$, and
\begin{equation*}
    \partial_t g_2 + (2G_2-s_m^{-1}\sin(2({\textstyle\frac{\pi}{m}}-\theta))m[g_1](t)) \partial_\theta g_2 = 0.
\end{equation*}
As $M[g_2](t) \ge 0$, \Cref{lem:single:stepfn:perturbed:fwd} implies that 
\begin{equation}
\label{ineq:g_1:rightendpoint-decay}
      \sup \operatorname{supp} g_1(t) \lesssim \frac{1}{1+t}, \quad \|g_1(t)\|_{L^1} \lesssim \frac{1}{1+t} , \quad m[g_1](t) \lesssim \frac{1}{(1+t)^2} 
\end{equation}
for $t \ge 0$.

We will now show that $g_1$ and $g_2$ are well approximated by $\bar{g}_1$ and $\bar{g}_2$. Choose $C_0$, depending
only on $h_2$ and $m$, so that
\[
 M[\bar g_2](t)\le \frac{C_0}{(1+T-t)^2}.
\]
Since $M[g_2](0)=M[\bar g_2](0)$, define
\[
 T_*:=\sup\left\{\tau\in[0,T]:
 M[g_2](t)\le\frac{2C_0}{(1+T-t)^2}
 \text{ for every }0\le t\le\tau\right\}.
\]
The equality at $t=0$ and the bound for $M[\bar g_2](0)$ show that $T_*>0$. On $[0,T_*]$, \Cref{lem:g_2-barg_2} below gives, for $T$ sufficiently large,
\begin{equation}
\label{eq:g2-distance-closed}
 \mathrm d(g_2(t),\bar g_2(t))+
 \|g_2(t)-\bar g_2(t)\|_{L^1}
 \le \frac{CC_0\log^2(2+T)}{(2+T)(1+T-t)^2}.
\end{equation}
Both $g_2$ and $\bar{g}_2$ are consequently supported at distance $O((1+T-t)^{-1})$ from $\frac{\pi}{m}$. Since the weight $\sin(2(\frac{\pi}{m}-\theta))$ is
$O((1+T-t)^{-1})$ on
this region, comparison of corresponding endpoint integrals gives
\[
 |M[g_2](t)-M[\bar g_2](t)|
 \le \frac{C}{1+T-t}\mathrm d(g_2(t),\bar g_2(t)).
\]
Together with \eqref{eq:g2-distance-closed}, this yields
\begin{equation}
\label{calc:M[g_2]}
 M[g_2](t)
 \le\frac{C_0}{(1+T-t)^2}
 \left(1+C\frac{\log^2(2+T)}{2+T}\right)
 \le\frac{3C_0}{2(1+T-t)^2}
\end{equation}
on $[0,T_*]$, once $T$ is large. Continuity therefore implies $T_*=T$.

It follows from \eqref{eq:g2-distance-closed} and \eqref{bar-g_2:bounds} that
\begin{equation}
\label{calc:g2L1}
 \|g_2(t)\|_{L^1}\lesssim (1+T-t)^{-1},\qquad 0\le t\le T.
\end{equation}
Furthermore, \Cref{lem:g_1-barg_1}, \eqref{calc:M[g_2]}, and
\eqref{eq:inner-integral-bound-oldstyle} below give
\[
 \sup_{0\le t\le T}\|g_1(t)-\bar g_1(t)\|_{L^1}\lesssim (1+T)^{-1}.
\]
Consequently,
\begin{equation}
\label{eq:two-packet-closeness-final}
 \sup_{0\le t\le T}\|g_2(t)-\bar g_2(t)\|_{L^1}
 +\sup_{0\le t\le T}\|g_1(t)-\bar g_1(t)\|_{L^1}
 \lesssim\frac{\log^2(2+T)}{2+T},
\end{equation}
which tends to zero as $T\to\infty$.
Note in particular that
\begin{equation*}
\begin{aligned}
 \|g(0) - h_1\|_{L^1} &=  \|g_2(0)\|_{L^1}=\|\bar{g}_2(0)\|_{L^1}\lesssim \frac{1}{1+T}
\end{aligned}
\end{equation*}
and
\begin{equation*}
\begin{aligned}
 \|g(T) - h_2\|_{L^1} &\le       \|g_2(T) - h_2\|_{L^1} + \|g_1(T)\|_{L^1}\\
 &= \|g_2(T) - \bar{g}_2(T)\|_{L^1} + \|g_1(T)\|_{L^1}\\
 &\le \sup_{0\le t\le T}\|g_2(t) - \bar{g}_2(t)\|_{L^1}+ \|g_1(T)\|_{L^1}\\ &\lesssim \frac{C_0\log^2(1+T)}{(1+T)} + \frac{1}{1+T}
\end{aligned}
\end{equation*}
Let $\varepsilon > 0$ be given. If $T$ is sufficiently large depending only on $h_1, h_2, m, \varepsilon$, we obtain
\begin{equation}
\label{g2:terminal}
     \|g(T) - h_2\|_{L^1}  \le \varepsilon, \quad  \|g(0) - h_1\|_{L^1} \le\varepsilon
\end{equation}
Collecting \eqref{g2:init}, \eqref{ineq:g_1:rightendpoint-decay} and \eqref{g2:terminal}, we obtain a solution to \Cref{auxiliary-question}.
We record this and the estimates \eqref{ineq:g_1:rightendpoint-decay}, \eqref{calc:M[g_2]}, and \eqref{calc:g2L1} below.
\begin{proposition}[Extension Lemma]
\label{lem:insertion}
    Let $0\le h_1,h_2\le1$ be given step functions supported in $(0, \frac{\pi}{m})$. Given $\varepsilon > 0$, there exists $T_0=T_0(h_1,h_2, \varepsilon)$ such that for each $T \ge T_0$, there is a solution $g$ of \eqref{eqn:Euler:reduced:transport}-\eqref{eqn:Euler:reduced:BSlaw} with the following properties.
    \begin{enumerate}
        \item $\operatorname{supp} g(t) \subset (0, \frac{\pi}{m})$ and  \begin{equation*} g(t,\theta) = g_1(t,\theta) + g_2(t, \theta)
         \end{equation*}
         with $ \operatorname{supp} g_1(t)$ lying to the left of $\operatorname{supp} g_2(t)$ and in particular $g_1, g_2$ have disjoint supports for all time. 

         \item 
         \begin{equation*}
             \|g_1(t)\|_{L^1} \lesssim \frac{1}{1+t}, \quad m[g_1](t) \lesssim \frac{1}{(1+t)^2}
         \end{equation*}
         for $t \ge 0$ and 
         \begin{equation*}
              \|g_2(t)\|_{L^1} \lesssim \frac{1}{1+T-t}, \quad M[g_2](t) \lesssim \frac{1}{(1+T-t)^2}
         \end{equation*}
         for $0 \le t \le T$.
        \item \begin{equation}
        \label{ineq:closeness-of-gs}
    \sup_{0\le t\le T}\|g_2(t) - \bar{g}_2(t)\|_{L^1}+\sup_{0\le t\le T}\|g_1(t) - \bar{g}_1(t)\|_{L^1} \lesssim \frac{\log^2(1+T)}{(1+T)}
\end{equation}
where $\bar{g}_1$ is the solution to 
\begin{equation*}
 \begin{aligned}
        \partial_t \bar{g}_1 + 2\bar{G}_1 \partial_\theta \bar{g}_1 &= 0 \\
                \bar{G}_1(t, \theta) &= \int_0^{\frac{\pi}{m}}K_m(\theta, \zeta)\bar{g}_1(t,\zeta) \ \mathrm{d}\zeta, \\
                \bar{g}_1(0, \theta) &= h_1
 \end{aligned}
\end{equation*}
and $\bar{g}_2$ is the solution of
\begin{equation*}
 \begin{aligned}
        \partial_t \bar{g}_2 + (2\bar{G}_2 - s_m^{-1}\sin(2({\textstyle\frac{\pi}{m}}-\theta))m[\bar{g}_1](t))\partial_\theta \bar{g}_2 &= 0 \\
                \bar{G}_2(t, \theta) &= \int_0^{\frac{\pi}{m}}K_m(\theta, \zeta)\bar{g}_2(t,\zeta) \ \mathrm{d}\zeta, \\
                \bar{g}_2(T, \theta) &= h_2
\end{aligned}
\end{equation*}
        \item  \begin{equation*}
            g_1(0, \theta) = h_1, \quad \|g_2(0)\|_{L^1} \le \varepsilon, \quad \|g_1(T)\|_{L^1} \le \varepsilon
        \end{equation*}
        and 
        \begin{equation*}
            \|g(T, \theta) - h_2\|_{L^1} \le \varepsilon,\quad\|g(0, \theta) - h_1\|_{L^1} \le \varepsilon
        \end{equation*}
\end{enumerate}
\end{proposition}
\begin{remark}
    It is useful to view \Cref{lem:insertion} as an ``extension lemma'' in the following sense. We think of $g_1$ with the desired behavior as already given, and append $g_2$ to it without affecting $g_1$ ``too much''. Of course, this extension is performed nonlinearly, which is why the proof is somewhat involved. This point of view is useful because the procedure can be iterated; we take this up in the next section.
\end{remark}

\begin{remark}
By time translation invariance, we can replace the times $0$ and $T$ in \Cref{lem:insertion} by the times $t_1$ and $t_2 = t_1+T$.
\end{remark}

\subsubsection{Estimate of $\bar{g}_1 - g_1$}
Note that both $g_1$ and $\bar{g}_1$ have the same topology (i.e. same arrangement of steps for all time, both being deformations of $h_1$) and their endpoints begin initially at the same location. Suppose momentarily that the endpoint trajectories of  $g_1$ are 
 \begin{equation*}
    0 < a_1(t) < b_1(t) < a_2(t) < b_2(t) < \cdots < a_{n_1}(t) < b_{n_1}(t) < \frac{\pi}{m}
\end{equation*}
while those of $\bar{g}_1$ are
 \begin{equation*}
    0 < \bar{a}_1(t) < \bar{b}_1(t) < \bar{a}_2(t) < \bar{b}_2(t) < \cdots < \bar{a}_{n_1}(t) < \bar{b}_{n_1}(t) < \frac{\pi}{m}
\end{equation*}
We will next obtain a bound on 
\begin{equation}
\label{def:dist-function}
   \mathrm{d}(g_1(t), \bar{g}_1(t)):=\max_{j}\{|\bar{a}_j(t)-a_j(t)| + |\bar{b}_j(t)-b_j(t)|\}.
\end{equation}
In the following arguments, the observation mentioned in \Cref{rmk:non-endpoint:decay}, or more specifically, the integrability of endpoint trajectories to the left of the last one, is crucial to avoid growth at a Gr\"{o}nwall step (see in particular \eqref{ineq:D(t):quadratic-est}).  
\begin{lemma}
\label[lemma]{lem:g_1-barg_1}
The following estimate holds, where the implicit constants depend on $h_1$.
    \begin{equation*}
        \mathrm{d}(g_1(t), \bar{g}_1(t))= \max_{j}\{|\bar{a}_j(t)-a_j(t)| + |\bar{b}_j(t)-b_j(t)|\}\lesssim\int_0^t \frac{M[g_2](s)}{1+s} \ \mathrm{d}s
    \end{equation*}
Consequently
\begin{equation*}
    \|g_1(t) - \bar{g}_1(t)\|_{L^1} \lesssim \int_0^t \frac{M[g_2](s)}{1+s} \ \mathrm{d}s
\end{equation*}
\end{lemma}

\begin{proof}
Since $\partial_t g_1 + (2G_1-s_m^{-1}\sin(2\theta) M[g_2](t)) \partial_\theta g_1 = 0$ and $ \partial_t \bar{g}_1 + 2\bar{G}_1 \partial_\theta \bar{g}_1 = 0$, the trajectories of left endpoints satisfy
    \begin{equation*}
    \begin{aligned}
       a_k'(t) &= \frac{1}{2s_m}\sin(2({\textstyle\frac{\pi}{m}}-a_k(t)))\sum_{j=1}^{k-1} c_{1,j}\{\cos(2b_j(t)) - \cos(2a_j(t))\} \\
            & \ - \frac{1}{2s_m}\sin\left(2a_k(t)\right)\sum_{j=k}^{n_1} c_{1,j}\left\{\cos(2({\textstyle\frac{\pi}{m}}-b_j(t))) - \cos(2({\textstyle\frac{\pi}{m}}-a_j(t)))\right\} \\
            &-\frac{1}{s_m}\sin(2a_k(t))M[g_2](t)
    \end{aligned}
\end{equation*}
and
\begin{equation*}
    \begin{aligned}
       \bar{a}_k'(t) &= \frac{1}{2s_m}\sin(2({\textstyle\frac{\pi}{m}}-\bar{a}_k(t)))\sum_{j=1}^{k-1} c_{1,j}\{\cos(2\bar{b}_j(t)) - \cos(2\bar{a}_j(t))\} \\
            & \ - \frac{1}{2s_m}\sin\left(2\bar{a}_k(t)\right)\sum_{j=k}^{n_1} c_{1,j}\left\{\cos(2({\textstyle\frac{\pi}{m}}-\bar{b}_j(t))) - \cos(2({\textstyle\frac{\pi}{m}}-\bar{a}_j(t)))\right\}
    \end{aligned}
\end{equation*}

We introduce some notation below for concision. Define
\begin{equation*}
    \begin{aligned}
        V_{a_k}(x_1,\dots, x_{n_1},y_1, y_2, \dots, y_{n_1}) &:= \frac{1}{2s_m}\sin(2({\textstyle\frac{\pi}{m}}-x_k))\sum_{j=1}^{k-1} c_{1,j}\{\cos(2y_j) - \cos(2x_j)\} \\
            & \ - \frac{1}{2s_m}\sin\left(2x_k\right)\sum_{j=k}^{n_1} c_{1,j}\left\{\cos(2({\textstyle\frac{\pi}{m}}-y_j)) - \cos(2({\textstyle\frac{\pi}{m}}-x_j))\right\},
    \end{aligned}
\end{equation*}
and set $V_a :=(V_{a_1}, \dots, V_{a_{n_1}})$, $A(t) := (a_1(t), \dots, a_{n_1}(t))$, $\bar{A}(t) = (\bar{a}_1(t), \dots, \bar{a}_{n_1}(t))$, $B(t) := (b_1(t), \dots, b_{n_1}(t))$, $\bar{B}(t) = (\bar{b}_1(t), \dots, \bar{b}_{n_1}(t))$. Define also 
\begin{equation*}
    S(x_1,\dots, x_{n_1}) = s_m^{-1}(\sin(2x_1), \dots, \sin(2x_{n_1}))
\end{equation*}
Then we may write
\begin{equation*}
    \begin{aligned}
        A'(t) &= V_a(A(t), B(t)) - M[g_2](t) S(A(t)) \\
        \bar{A}'(t) &= V_a(\bar{A}(t),\bar{B}(t) )
    \end{aligned}
\end{equation*}
while $A(0) = \bar{A}(0)$.
Similarly, recalling the right-end point ODEs
\begin{equation*}
    \begin{aligned}
      b_k'(t)&= \frac{1}{2s_m}\sin(2({\textstyle\frac{\pi}{m}}-b_k(t)))\sum_{j=1}^{k} c_{1,j}\{\cos(2b_j(t)) - \cos(2a_j(t))\} \\
            & \ - \frac{1}{2s_m}\sin(2b_k(t))\sum_{j=k+1}^{n_1} c_{1,j}\left\{\cos(2({\textstyle\frac{\pi}{m}}-b_j(t))) - \cos(2({\textstyle\frac{\pi}{m}}-a_j(t)))\right\}\\
            &-\frac{1}{s_m}\sin(2b_k(t))M[g_2](t)
    \end{aligned}
\end{equation*}
and the corresponding equation for $\bar{b}_k'(t)$, we can write 
\begin{equation*}
    \begin{aligned}
        B'(t) &= V_b(A(t), B(t)) - M[g_2](t) S(B(t)) \\
        \bar{B}'(t) &= V_b(\bar{A}(t), \bar{B}(t) )
    \end{aligned}
\end{equation*}
where $V_b = (V_{b_1}, \dots, V_{b_{n_1}})$ and
\begin{equation*}
\begin{aligned}
        V_{b_k}(x_1,\dots, x_{n_1},y_1, y_2, \dots, y_{n_1}) &= \frac{1}{2s_m}\sin(2({\textstyle\frac{\pi}{m}}-y_k))\sum_{j=1}^{k} c_{1,j}\{\cos(2y_j) - \cos(2x_j)\} \\
            & \ - \frac{1}{2s_m}\sin(2y_k)\sum_{j=k+1}^{n_1} c_{1,j}\left\{\cos(2({\textstyle\frac{\pi}{m}}-y_j)) - \cos(2({\textstyle\frac{\pi}{m}}-x_j))\right\}\\ 
\end{aligned}
\end{equation*}
Set 
\begin{equation}
\label{def:simplifying}
    \begin{aligned}
     V = (V_a, V_b), \quad Z =    (A,  B), \quad \bar{Z} = (\bar{A}, \bar{B}),\quad D(t) = \bar{Z}-Z,
    \end{aligned}
\end{equation}
and
\begin{equation}
\label{def:simplifying:sine_vec}
    \mathcal{S}(x_1,\dots,x_{n_1}, y_1,\dots, y_{n_1}) = (S(x_1,\dots, x_{n_1}), S(y_1, \dots, y_{n_1}))
\end{equation}
Then we have

\begin{equation}
\label{eqn:simplified:ZZ_bar}
    Z'(t) = V(Z(t)) - M[g_2](t)\mathcal{S}(Z(t)), \quad  \bar{Z}'(t) = V(\bar{Z}(t))
\end{equation}
and
\begin{equation}
\label{eqn:D(t)}
\begin{aligned}
        D'(t) &=V(\bar{Z}(t))- V(Z(t)) +M[g_2](t) \mathcal{S}(Z(t))  \\
        &= Q(t)D(t) + M[g_2](t) \mathcal{S}(Z(t))
\end{aligned} 
\end{equation}
where
\begin{equation*}
\begin{aligned}
    Q(t) &= \int_0^1 \nabla V((\bar{A}(t), \bar{B}(t)) - s D(t)) \ \mathrm{d}s
\end{aligned}
\end{equation*}
We claim that
\begin{equation}
\label{ineq:D(t):quadratic-est}
    D(t)^{\top}Q(t)D(t) \lesssim \Lambda(t)|D(t)|^2
\end{equation}
where \[\int_0^\infty\Lambda(t) \mathrm{d}t < \infty.\]
We relegate the proof of the claim \eqref{ineq:D(t):quadratic-est} to the end of the section below. 
Assuming momentarily \eqref{ineq:D(t):quadratic-est}, we have
\begin{equation*}
        \begin{aligned}
       \frac{1}{2} \frac{\mathrm{d}}{\mathrm{d}t}|D(t)|^2 &= D(t)^{\top}Q(t)D(t) + M[g_2](t) D(t)^{\top}\mathcal{S}(Z(t)) \\
        &\lesssim \Lambda(t)|D(t)|^2 + M[g_2](t) |D(t)|(|S(B(t))|+|S(A(t))|) \\
        &\lesssim  \Lambda(t)|D(t)|^2 + M[g_2](t) \frac{|D(t)|}{1+t}
    \end{aligned}
\end{equation*}
since by \eqref{ineq:g_1:rightendpoint-decay}
\begin{equation*}
    |S(B(t))| \lesssim |B(t)| \lesssim (1+t)^{-1}, \quad  |S(A(t))| \lesssim |A(t)| \lesssim (1+t)^{-1}
\end{equation*}
Hence
\begin{equation*}
        \begin{aligned}
        \frac{\mathrm{d}}{\mathrm{d}t}|D(t)|
        &\lesssim  \Lambda(t)|D(t)| +  \frac{M[g_2](t)}{1+t}
    \end{aligned}
\end{equation*}
and by Gr\"{o}nwall's inequality, noting $|D(0)| = 0$, 
\begin{equation}
    |D(t)| \lesssim e^{\int_0^\infty \Lambda(s) \ \mathrm{d}s}\int_0^t \frac{M[g_2](s)}{1+s} \ \mathrm{d}s
\end{equation}
as desired. Hence
\begin{equation*}
     \max_{j}\{|\bar{a}_j(t)-a_j(t)| + |\bar{b}_j(t)-b_j(t)|\}\lesssim\int_0^t \frac{M[g_2](s)}{1+s} \ \mathrm{d}s
\end{equation*}
and
\begin{equation*}
\begin{aligned}
        \|g_1(t) - \bar{g}_1(t)\|_{L^1}  &= \left\|\sum_{j=1}^{n_1} c_{1,j}(\mathbf{1}_{(a_{j}(t), b_{j}(t)]}-\mathbf{1}_{(\bar{a}_{j}(t), \bar{b}_{j}(t)]})\right\|_{L^1} \\
        &\lesssim  \max_{j}\{|\bar{a}_j(t)-a_j(t)| + |\bar{b}_j(t)-b_j(t)|\}\lesssim\int_0^t \frac{M[g_2](s)}{1+s} \ \mathrm{d}s.
\end{aligned}
\end{equation*}
\end{proof}

\begin{proofof}{\eqref{ineq:D(t):quadratic-est}}
\label{proof:D(t):quadratic-est}
    We compute $\nabla V$. Let us write $x = (x_1, \dots, x_{n_1}), y = (y_1, \dots, y_{n_1})$ and we say that $(x,y)$ is ordered provided $x_1 < y_1 < x_2 < y_2 < \dots < x_{n_1} < y_{n_1}$. We henceforth assume that $(x,y)$ is ordered. 

Write $\nabla V$ in the block form
\begin{equation*}
    \nabla V = \begin{bmatrix}
        \nabla_x V_a & \nabla_y V_a \\\nabla_x V_b & \nabla_y V_b
    \end{bmatrix}
\end{equation*}
i.e. $ (\nabla_x V_a)_{kj} =\frac{\partial V_{a_k}}{\partial x_j}$ and so on. We compute the following
\begin{enumerate}
\item \underline{ Diagonal terms in  $\nabla_y V_a$ and $\nabla_x V_b$}

Note that
\begin{equation*}
    \begin{aligned}
        \frac{\partial V_{a_k}}{\partial y_k} &= -c_{1,k}\frac{1}{s_m}\sin(2x_k)\sin(2({\textstyle\frac{\pi}{m}}-y_k)) \\
        \frac{\partial V_{b_k}}{\partial x_k} &= c_{1,k}\frac{1}{s_m}\sin(2x_k)\sin(2({\textstyle\frac{\pi}{m}}-y_k))
    \end{aligned}
\end{equation*}
i.e. $\frac{\partial V_{a_k}}{\partial y_k} = - \frac{\partial V_{b_k}}{\partial x_k}$.

This says that the diagonal of $\nabla_y V_a$ and that of $\nabla_x V_b$ have opposite signs. 
\item \underline{Off-diagonal terms in each block}

Compute for $j < k$,
\begin{equation*}
    \begin{aligned}
         \frac{\partial V_{a_k}}{\partial x_j} &= \frac{c_{1,j}}{s_m}\sin(2({\textstyle\frac{\pi}{m}}-x_k))\sin(2x_j), \quad  \frac{\partial V_{a_k}}{\partial y_j} =-\frac{c_{1,j}}{s_m}\sin(2({\textstyle\frac{\pi}{m}}-x_k))\sin(2y_j)
    \end{aligned}
\end{equation*}
and
\begin{equation*}
    \begin{aligned}
         \frac{\partial V_{b_k}}{\partial x_j} &= \frac{c_{1,j}}{s_m}\sin(2({\textstyle\frac{\pi}{m}}-y_k))\sin(2x_j), \quad  \frac{\partial V_{b_k}}{\partial y_j} =-\frac{c_{1,j}}{s_m}\sin(2({\textstyle\frac{\pi}{m}}-y_k))\sin(2y_j)
    \end{aligned}
\end{equation*}
 Thus,
\begin{equation*}
    \left|\frac{\partial V_{a_k}}{\partial x_j}\right| + \left| \frac{\partial V_{a_k}}{\partial y_j}\right| \le C(\sin(2x_j)+\sin(2y_j)) \le C y_j \le Cx_{n_1}
\end{equation*}
and likewise
\begin{equation*}
    \left|\frac{\partial V_{b_k}}{\partial x_j}\right| + \left| \frac{\partial V_{b_k}}{\partial y_j}\right| \le Cx_{n_1}
\end{equation*}

For $j > k$, we have
\begin{equation*}
    \begin{aligned}
         \frac{\partial V_{a_k}}{\partial x_j} &= \frac{c_{1,j}}{s_m}\sin(2({\textstyle\frac{\pi}{m}}-x_j))\sin(2x_k), \quad  \frac{\partial V_{a_k}}{\partial y_j} =-\frac{c_{1,j}}{s_m}\sin(2({\textstyle\frac{\pi}{m}}-y_j))\sin(2x_k) \\
        \frac{\partial V_{b_k}}{\partial x_j} &= \frac{c_{1,j}}{s_m}\sin(2({\textstyle\frac{\pi}{m}}-x_j))\sin(2y_k), \quad  \frac{\partial V_{b_k}}{\partial y_j} =-\frac{c_{1,j}}{s_m}\sin(2({\textstyle\frac{\pi}{m}}-y_j))\sin(2y_k)
    \end{aligned}
\end{equation*}
and hence also
\begin{equation*}
    \left|\frac{\partial V_{a_k}}{\partial y_j}\right| + \left| \frac{\partial V_{a_k}}{\partial x_j}\right| \le Cx_{n_1},  \quad    \left|\frac{\partial V_{b_k}}{\partial y_j}\right| + \left| \frac{\partial V_{b_k}}{\partial x_j}\right| \le Cx_{n_1}
\end{equation*}

Altogether, we have that the off-diagonal terms in each of the blocks $ \nabla_x V_a, \nabla_y V_a,\nabla_x V_b, \nabla_y V_b$ are bounded in absolute value by $Cx_{n_1}$. 
\item \underline{Main-diagonal terms}

For $j=k$, compute
\begin{equation*}
    \begin{aligned}
        \frac{\partial V_{a_k}}{\partial x_k} &= -\frac{1}{s_m}\cos(2({\textstyle\frac{\pi}{m}}-x_k))\sum_{j=1}^{k-1} c_{1,j}\{\cos(2y_j) - \cos(2x_j)\} \\
         & \ - \frac{1}{s_m}\cos\left(2x_k\right)\sum_{j=k}^{n_1} c_{1,j}\left\{\cos(2({\textstyle\frac{\pi}{m}}-y_j)) - \cos(2({\textstyle\frac{\pi}{m}}-x_j))\right\}\\
         &+\frac{c_{1,k}}{s_m}\sin\left(2x_k\right)\sin(2({\textstyle\frac{\pi}{m}}-x_k)) \\
    \end{aligned}
\end{equation*}

\begin{equation*}
    \begin{aligned}
         \frac{\partial V_{b_k}}{\partial y_k} &= -\frac{1}{s_m}\cos(2({\textstyle\frac{\pi}{m}}-y_k))\sum_{j=1}^{k} c_{1,j}\{\cos(2y_j) - \cos(2x_j)\} \\
         &-\frac{c_{1,k}}{s_m}\sin(2({\textstyle\frac{\pi}{m}}-y_k))\sin(2y_k)\\
            & \ - \frac{1}{s_m}\cos(2y_k)\sum_{j=k+1}^{n_1} c_{1,j}\left\{\cos(2({\textstyle\frac{\pi}{m}}-y_j)) - \cos(2({\textstyle\frac{\pi}{m}}-x_j))\right\}\\ 
    \end{aligned}
\end{equation*}

Let $x_+ := \max\{x ,0\}$. Then by the ordering assumption,
\begin{equation*}
    \begin{aligned}
        \left(\frac{\partial V_{a_k}}{\partial x_k}\right)_+ &\le -\frac{1}{s_m}\cos(2({\textstyle\frac{\pi}{m}}-x_k))\sum_{j=1}^{k-1} c_{1,j}\{\cos(2y_j) - \cos(2x_j)\} \\
        &+\frac{c_{1,k}}{s_m}\sin\left(2x_k\right)\sin(2({\textstyle\frac{\pi}{m}}-x_k)) \\
        &=\frac{2}{s_m}\cos(2({\textstyle\frac{\pi}{m}}-x_k))\sum_{j=1}^{k-1} c_{1,j}\{\sin^2(y_j) - \sin^2(x_j)\} \\
         &+\frac{c_{1,k}}{s_m}\sin\left(2x_k\right)\sin(2({\textstyle\frac{\pi}{m}}-x_k)) \\
         &\le C\sum_{j=1}^{k-1} \sin^2(y_j)+2\sin(2x_k)\\
         &\le C\sum_{j=1}^{k-1} y_j^2 +4x_k \lesssim x_{n_1}+x_{n_1}^2 \lesssim x_{n_1}\\
    \end{aligned}
\end{equation*}
and
\begin{equation*}
\begin{aligned}
        \left(\frac{\partial V_{b_k}}{\partial y_k}\right)_+ &\le   -\frac{1}{s_m}\cos(2({\textstyle\frac{\pi}{m}}-y_k))\sum_{j=1}^{k} c_{1,j}\{\cos(2y_j) - \cos(2x_j)\} \\
        &\le C\sum_{j=1}^{k} y_j^2 \lesssim y_{n_1}^2
\end{aligned}
\end{equation*}
This shows that 
the terms in the diagonal of $\nabla V$ have positive parts that are bounded by $\lesssim x_{n_1} + y_{n_1}^2$ in the region where $(x,y)$ is ordered.
\end{enumerate}

Denote by $Q_{\operatorname{sym}}(t)$ the symmetric part of $Q(t)$. Thus
\begin{equation*}
    \begin{aligned}
D(t)^{\top}Q(t)D(t)
        &= D(t)^{\top}Q_{\operatorname{sym}}(t)D(t) 
    \end{aligned}
\end{equation*}

Now, from the definition of $Q(t)$, we have
\begin{equation*}
    Q_{\operatorname{sym}}(t) =\int_0^1 (\nabla V)_{\operatorname{sym}}((\bar{A}(t), \bar{B}(t)) - s D(t)) \ \mathrm{d}s
\end{equation*}
and the terms in the vector $(\bar{A}(t), \bar{B}(t)) - s D(t)$ are ordered, being a convex combination of ordered vectors $(\bar{A}(t), \bar{B}(t))$ and $({A}(t), {B}(t))$. 
Write
 \begin{equation*}
     \begin{aligned}
         D(t)^{\top}Q_{\operatorname{sym}}(t)D(t) &= \sum_{j=1}^{n_1} \left[\alpha_{jj}(t)\{\bar{a}_j(t)-a_j(t)\}^2 + \beta_{jj}(t)\{\bar{b}_j(t)-b_j(t)\}^2\right] \\
         &+ 2\sum_{1\le j < k\le 2n_1} (Q_{\operatorname{sym}}(t))_{jk}D_j(t)D_k(t)
     \end{aligned}
 \end{equation*}
where $\alpha_{jj}(t)$ and $\beta_{jj}(t)$ satisfy the following.
The main diagonal computation above (item 3.) implies ($\Pi_{x_{n_1}} =$ projection onto the $x_{n_1}$ coordinate, and similarly for $\Pi_{y_{n_1}}$) 
\begin{equation*}
\begin{aligned}
        \alpha_{jj}(t)_+&\lesssim \int_0^1 \Pi_{x_{n_1}}(\bar{Z}(t) - s D(t)) + \{\Pi_{y_{n_1}}(\bar{Z}(t) - s D(t))\}^2 \  \mathrm{d}s\\
        &=  \int_0^1 \bar{a}_{n_1}(t)- s(\bar{a}_{n_1}(t) - a_{n_1}(t)) + \{\bar{b}_{n_1}(t)- s(\bar{b}_{n_1}(t) - b_{n_1}(t))\}^2\  \mathrm{d}s\\
        &\lesssim \bar{a}_{n_1}(t) +a_{n_1}(t)  + \bar{b}_{n_1}(t)^2 + b_{n_1}(t)^2
\end{aligned}
\end{equation*}
and likewise,
\begin{equation*}
     \beta_{jj}(t)_+\lesssim   \bar{a}_{n_1}(t) +a_{n_1}(t)  + \bar{b}_{n_1}(t)^2 + b_{n_1}(t)^2
\end{equation*}
 By \Cref{prop:single:stepfn} and \Cref{lem:single:stepfn:perturbed:fwd}, $\Lambda(t) :=  \bar{a}_{n_1}(t) +a_{n_1}(t)  + \bar{b}_{n_1}(t)^2 + b_{n_1}(t)^2$ is integrable. Therefore,
\begin{equation*}
    \sum_{j=1}^{n_1} \alpha_{jj}(t)\{\bar{a}_j(t)-a_j(t)\}^2 + \beta_{jj}(t)\{\bar{b}_j(t)-b_j(t)\}^2  \lesssim \Lambda(t)|D(t)|^2
\end{equation*}
 
Finally, consider 
\begin{equation*}
    \sum_{1\le j < k\le 2n_1} (Q_{\operatorname{sym}}(t))_{jk}D_j(t)D_k(t)
\end{equation*}

According to items 1. and 2. in the computation of $\nabla V$ above, off-diagonal terms in $Q_{\operatorname{sym}}(t)$ are either $0$ (due to item 1.) or bounded by (due to item 2.)
\begin{equation*}
    \lesssim \int_0^1 \Pi_{x_{n_1}}(\bar{Z}(t) - s D(t))  \  \mathrm{d}s \lesssim  \bar{a}_{n_1}(t) + a_{n_1}(t) \le \Lambda(t)
\end{equation*}

Thus

\begin{equation*}
\begin{aligned}
    \sum_{1\le j < k\le 2n_1} (Q_{\operatorname{sym}}(t))_{jk}D_j(t)D_k(t) &\lesssim \Lambda(t)|D(t)|^2
    \end{aligned}
\end{equation*}
This completes the proof of the claim \eqref{ineq:D(t):quadratic-est}.
\end{proofof}

\subsubsection{Estimate of $\bar{g}_2-g_2$}
Suppose that the endpoint trajectories of $g_2$ and $\bar g_2$ are, respectively,
\[
 0<a_1<b_1<\cdots<a_{n_2}<b_{n_2}<\frac\pi m,
 \qquad
 0<\bar a_1<\bar b_1<\cdots<\bar a_{n_2}<\bar b_{n_2}<\frac\pi m.
\]
We use the endpoint distance
\[
 \mathrm d(g_2(t),\bar g_2(t))
 :=\max_j\bigl\{|\bar a_j(t)-a_j(t)|+|\bar b_j(t)-b_j(t)|\bigr\}.
\]

\begin{lemma}
\label[lemma]{lem:g_2-barg_2}
Define
\begin{equation}
\label{eq:forcing-functional}
 \mathcal F_T(t):=\int_0^t\frac{1+T-s}{1+s}
 \left(\int_0^s\frac{M[g_2](\sigma)}{1+\sigma}\,\mathrm d\sigma\right)\mathrm ds.
\end{equation}
There are constants $C$ and $\varepsilon_0>0$, depending only on $h_1,h_2$
and $m$, with the following property. For every $0<\tau\le T$, if
$\mathcal F_T(\tau)\le\varepsilon_0$, then
\begin{equation}
\label{eq:g2-weighted-stability}
 \mathrm d(g_2(t),\bar g_2(t))
 \le \frac{C}{(1+T-t)^2}\mathcal F_T(t),\qquad 0\le t\le \tau.
\end{equation}
The same estimate holds with the left-hand side replaced by
$\|g_2(t)-\bar g_2(t)\|_{L^1}$.

In particular, if, on an interval $0\le t\le T_*\le T$,
\begin{equation}
\label{eq:g2-moment-bootstrap-lemma}
 M[g_2](t)\le\frac{2C_0}{(1+T-t)^2},
\end{equation}
then, for $T$ sufficiently large,
\begin{equation}
\label{eq:g2-bootstrap-stability}
 \mathrm d(g_2(t),\bar g_2(t))
 +\|g_2(t)-\bar g_2(t)\|_{L^1}
 \le \frac{CC_0\log^2(2+T)}{(2+T)(1+T-t)^2},\qquad 0\le t\le T_*.
\end{equation}
\end{lemma}

\begin{proof}
Recall \[\partial_t g_2 + (2G_2-s_m^{-1}\sin(2({\textstyle\frac{\pi}{m}}-\theta))m[g_1](t)) \partial_\theta g_2 = 0\] and \[\partial_t \bar{g}_2 + (2\bar{G}_2 - s_m^{-1}\sin(2({\textstyle\frac{\pi}{m}}-\theta))m[\bar{g}_1](t))\partial_\theta \bar{g}_2 = 0\] 
both with data at time $0$ equal to $\bar{g}_2(0)$. To estimate the endpoint differences $|\bar{a}_j(t)-a_j(t)| + |\bar{b}_j(t)-b_j(t)|$,adapt the notation of \Cref{lem:g_1-barg_1} (see in particular \eqref{def:simplifying}) in the obvious way: i.e. $\bar{Z}(t) =(\bar{A}(t), \bar{B}(t))$ is an $n_2 + n_2$ vector with left and right endpoints of $\bar{g}_2$ and $Z(t) = (A(t), B(t))$ has the analogous meaning for $g_2$. Define in the following 
\begin{equation*}
    \mathcal{S}_{\frac{\pi}{m}}(x_1,\dots, x_{n_2}, y_1,\dots,y_{n_2}) = \mathcal{S}({\textstyle\frac{\pi}{m}}-x_1, \dots, {\textstyle\frac{\pi}{m}}-x_{n_2}, {\textstyle\frac{\pi}{m}}-y_1, \dots, {\textstyle\frac{\pi}{m}}-y_{n_2})
\end{equation*}
where $\mathcal{S}$ is as defined in \eqref{def:simplifying:sine_vec}, but $n_1$ is replaced by $n_2$.
The analog of the equations \eqref{eqn:simplified:ZZ_bar} in the current case now reads
\begin{equation}
    Z'(t) = V(Z(t)) - m[g_1](t)\mathcal{S}_{\frac{\pi}{m}}(Z(t)), \quad  \bar{Z}'(t) = V(\bar{Z}(t))- m[\bar{g}_1](t)\mathcal{S}_{\frac{\pi}{m}}(\bar{Z}(t))
\end{equation}
with $V$ (and hence $Q$) structurally the same as in \Cref{lem:g_1-barg_1} but with $n_1$ replaced by $n_2$. Recall the following estimates on $\bar{Z}(t)$ for $0 \le t \le T$:
from \eqref{bar-g_2:bounds},
\begin{equation*}
    \frac{\pi}{m}-\bar{Z}_i(t) \lesssim \frac{1}{(1+T-t)}.
\end{equation*}

This motivates performing the change of variables
\begin{equation*}
    y_j(t) = \frac{\pi}{m} - a_{n_2+1-j}(t), \quad x_j(t) = \frac{\pi}{m} - b_{n_2+1-j}(t)
\end{equation*}
and
\begin{equation*}
\bar{y}_{n_2+1-j}(t) =\frac{\pi}{m} - \bar{a}_j(t),\quad \bar{x}_{n_2+1-j}(t) = \frac{\pi}{m} - \bar{b}_j(t).
\end{equation*}
Notice carefully that we are exchanging the roles of the left and right endpoints, which gives the ordering $x_1 < y_1 < x_2 < y_2 < \dots < x_{n_2} < y_{n_2}$ and leaves the structure of $V$ unchanged. Setting 
\[X(t) = (x_1(t),\dots, x_{n_2}(t), y_1(t),\dots, y_{n_2}(t))\] and \[\bar{X}(t) = (\bar{x}_1(t),\dots, \bar{x}_{n_2}(t), \bar{y}_1(t),\dots, \bar{y}_{n_2}(t)),\] we then have
\begin{equation}
    X'(t) = -V(X(t)) + m[g_1](t)\mathcal{S}(X(t)), \quad  \bar{X}'(t) = -V(\bar{X}(t))+ m[\bar{g}_1](t)\mathcal{S}(\bar{X}(t))
\end{equation}
Set $D(t) = \bar{X}(t) - X(t)$. Thus
\begin{equation}
\label{eqn:D(t):backwards}
    D'(t) = -Q(t)D(t) + (m[\bar{g}_1](t)-m[g_1](t))\mathcal{S}(\bar{X}(t)) + m[g_1](t)(\mathcal{S}(\bar{X}(t))-\mathcal{S}(X(t)) )
\end{equation}
and
\begin{equation}
    D(0) = 0
\end{equation}
with $Q$ here structurally identical to the one in  \Cref{lem:g_1-barg_1} above. Revisit the computation of $\nabla V$. Here in an abuse of notation, $c_{2,j}$ denote coefficients after the relabeling $c_{2,n_2 + 1 - k} \mapsto c_{2,k}$, though their exact values are immaterial and should cause no confusion.

\begin{enumerate}
\item \underline{Diagonal terms in  $\nabla_y V_a$ and $\nabla_x V_b$}

As in the proof of \Cref{lem:g_1-barg_1},
\begin{equation*}
 \frac{\partial V_{a_k}}{\partial y_k}
 =-\frac{c_{2,k}}{s_m}\sin(2x_k)
   \sin(2({\textstyle\frac{\pi}{m}}-y_k)),
 \qquad
 \frac{\partial V_{b_k}}{\partial x_k}
 =\frac{c_{2,k}}{s_m}\sin(2x_k)
   \sin(2({\textstyle\frac{\pi}{m}}-y_k)).
\end{equation*}
Thus the diagonal of $\nabla_yV_a$ is the negative of the diagonal of
$\nabla_xV_b$, and these entries cancel in the symmetric part. 
\item \underline{Off-diagonal terms in each block}

The calculation in \Cref{lem:g_1-barg_1}, with $n_1$ replaced by $n_2$
and $c_{1,j}$ replaced by $c_{2,j}$, shows that every off-diagonal entry
of the four blocks $\nabla_xV_a,\nabla_yV_a,\nabla_xV_b,\nabla_yV_b$ is
bounded in absolute value by $Cx_{n_2}$. 
\item \underline{Main-diagonal terms}

For $j=k$, we have
\begin{equation*}
    \begin{aligned}
        \frac{\partial V_{a_k}}{\partial x_k} &= -\frac{1}{s_m}\cos(2({\textstyle\frac{\pi}{m}}-x_k))\sum_{j=1}^{k-1} c_{2,j}\{\cos(2y_j) - \cos(2x_j)\} \\
         & \ - \frac{1}{s_m}\cos\left(2x_k\right)\sum_{j=k}^{n_2} c_{2,j}\left\{\cos(2({\textstyle\frac{\pi}{m}}-y_j)) - \cos(2({\textstyle\frac{\pi}{m}}-x_j))\right\}\\
         &+\frac{c_{2,k}}{s_m}\sin\left(2x_k\right)\sin(2({\textstyle\frac{\pi}{m}}-x_k)) \\
    \end{aligned}
\end{equation*}
We isolate the leading contribution from the $j=n_2$ term in the sum above. This term is
\begin{equation*}
\begin{aligned}
        &-\frac{1}{s_m}\cos\left(2x_k\right)c_{2,n_2}\left\{\cos(2({\textstyle\frac{\pi}{m}}-y_{n_2})) - \cos(2({\textstyle\frac{\pi}{m}}-x_{n_2}))\right\}\\ &\simeq -c_{2,n_2}\left\{\cos(2({\textstyle\frac{\pi}{m}}-y_{n_2})) - \cos(2({\textstyle\frac{\pi}{m}}-x_{n_2}))\right\}
\end{aligned}
\end{equation*}
for small $x_k$. Define
\begin{equation*}
    \lambda := \frac{c_{2,n_2}}{s_m}\left\{\cos(2({\textstyle\frac{\pi}{m}}-y_{n_2})) - \cos(2({\textstyle\frac{\pi}{m}}-x_{n_2}))\right\}
\end{equation*}
Thus
\begin{equation*}
    \begin{aligned}
        \frac{\partial V_{a_k}}{\partial x_k} + \lambda &= -\frac{1}{s_m}\cos(2({\textstyle\frac{\pi}{m}}-x_k))\sum_{j=1}^{k-1} c_{2,j}\{\cos(2y_j) - \cos(2x_j)\} \\
         & \ - \frac{1}{s_m}\cos\left(2x_k\right)\sum_{j=k}^{n_2-1} c_{2,j}\left\{\cos(2({\textstyle\frac{\pi}{m}}-y_j)) - \cos(2({\textstyle\frac{\pi}{m}}-x_j))\right\}\\
         &+ \lambda(1-\cos(2x_k))\\
         &+\frac{c_{2,k}}{s_m}\sin\left(2x_k\right)\sin(2({\textstyle\frac{\pi}{m}}-x_k)) \\
    \end{aligned}
\end{equation*}
and hence
\begin{equation*}
    \begin{aligned}
        \left|\frac{\partial V_{a_k}}{\partial x_k} + \lambda \right| &\lesssim \sum_{j=1}^k y_j^2 + \sum_{j=k}^{n_2-1} (y_j-x_j) + y_{n_2}x_k^2+ x_k\\
        &\lesssim y_{n_2}^2 + x_{n_2}
    \end{aligned}
\end{equation*}
due to the ordering and simple inequalities such as
\begin{equation*}
    1-\cos(2x_k) \lesssim x_k^2, \quad |\lambda| \lesssim y_{n_2} -x_{n_2} < y_{n_2}.
\end{equation*}
Similarly,
\begin{equation*}
    \begin{aligned}
         \frac{\partial V_{b_k}}{\partial y_k} &= -\frac{1}{s_m}\cos(2({\textstyle\frac{\pi}{m}}-y_k))\sum_{j=1}^{k} c_{2,j}\{\cos(2y_j) - \cos(2x_j)\} \\
         &-\frac{c_{2,k}}{s_m}\sin(2({\textstyle\frac{\pi}{m}}-y_k))\sin(2y_k)\\
            & \ - \frac{1}{s_m}\cos(2y_k)\sum_{j=k+1}^{n_2} c_{2,j}\left\{\cos(2({\textstyle\frac{\pi}{m}}-y_j)) - \cos(2({\textstyle\frac{\pi}{m}}-x_j))\right\}\\ 
    \end{aligned}
\end{equation*}
and for $k < n_2$, we can argue as above to get
\begin{equation*}
    \begin{aligned}
        \left|\frac{\partial V_{b_k}}{\partial y_k} + \lambda \right| \lesssim y_{n_2}^2 + x_{n_2}
    \end{aligned}
\end{equation*}
That leaves the case $k = n_2$. We have
\begin{equation*}
    \begin{aligned}
         \frac{\partial V_{b_{n_2}}}{\partial y_{n_2}} + \lambda &= -\frac{1}{s_m}\cos(2({\textstyle\frac{\pi}{m}}-y_{n_2}))\sum_{j=1}^{n_2-1} c_{2,j}\{\cos(2y_j) - \cos(2x_j)\} \\
       & -\frac{1}{s_m}\cos(2({\textstyle\frac{\pi}{m}}-y_{n_2})) c_{2,n_2}\{\cos(2y_{n_2}) - \cos(2x_{n_2})\} \\
         &-\frac{c_{2,n_2}}{s_m}\sin(2({\textstyle\frac{\pi}{m}}-y_{n_2}))\sin(2y_{n_2})\\
        &+  \frac{c_{2,n_2}}{s_m}\left\{\cos(2({\textstyle\frac{\pi}{m}}-y_{n_2})) - \cos(2({\textstyle\frac{\pi}{m}}-x_{n_2}))\right\} \\
         &= O(x_{n_2}^2) \\
         & -\frac{c_{2,n_2}}{s_m}\cos(2({\textstyle\frac{\pi}{m}}-y_{n_2})) \{\cos(2y_{n_2}) - \cos(2x_{n_2})\} \\
         &-\frac{c_{2,n_2}}{s_m}\sin(2({\textstyle\frac{\pi}{m}}-y_{n_2}))\sin(2y_{n_2})\\
         &+ \frac{c_{2,n_2}}{s_m}\left\{\cos(2({\textstyle\frac{\pi}{m}}-y_{n_2})) - \cos(2({\textstyle\frac{\pi}{m}}-x_{n_2}))\right\} \\
         &= O(x_{n_2}^2) + H(x_{n_2}, y_{n_2})\\
    \end{aligned}
\end{equation*}
where, as is easily checked,
\begin{equation*}
 H(0,0) = 0, \quad \partial_yH(0,0) = 0.
\end{equation*}
Hence by Taylor's theorem, $|H(x_{n_2}, y_{n_2})| \lesssim x_{n_2} + y_{n_2}^2$, since $x_{n_2}^2 \le y_{n_2}^2$ due to ordering. Thus
\begin{equation*}
     \left|\frac{\partial V_{b_{n_2}}}{\partial y_{n_2}} + \lambda\right| \lesssim x_{n_2} + y_{n_2}^2
\end{equation*}
Equivalently, the preceding estimates give
\begin{equation}
    -C(x_{n_2} + y_{n_2}^2)\operatorname{Id} \le (\nabla V)_{\operatorname{sym}} + \lambda(x_{n_2},y_{n_2})\operatorname{Id} \le C(x_{n_2} + y_{n_2}^2)\operatorname{Id},
\end{equation}
i.e.
\begin{equation}
    -(\nabla V)_{\operatorname{sym}} = \lambda(x_{n_2},y_{n_2})\operatorname{Id} + R, 
\end{equation}
with $\|R\| := \max_{ij}|R_{ij}| = O(x_{n_2} + y_{n_2}^2)$.
\end{enumerate}

We next prove the analogue of \eqref{ineq:D(t):quadratic-est}.  The refined
endpoint estimates \eqref{eq:backward-refined-endpoints} and
\eqref{eq:backward-sharp-endpoint}, applied to $\bar g_2$, give
\begin{equation}
\label{eq:reference-right-refined-oldstyle}
 \bar y_{n_2}(t)\simeq \frac{1}{1+T-t},\qquad
 \bar x_{n_2}(t)\lesssim \frac{1}{(1+T-t)^{1+\gamma}},
\end{equation}
and
\begin{equation}
\label{eq:reference-right-sharp-oldstyle}
 c_{2,n_2}\bar y_{n_2}(t)
 \le \frac{1}{1+T-t}
 +\frac{C}{(1+T-t)^{1+\gamma}}
\end{equation}
for some $\gamma>0$.  The assumption required there is satisfied because
$s_m^{-1}m[\bar g_1](t)$ has bounded $L^1(0,T)$ norm, uniformly in $T$, by
\Cref{prop:single:stepfn}.

In the equation for $D$, define, exactly as in \Cref{lem:g_1-barg_1},
\begin{equation*}
 Q(t)=\int_0^1\nabla V(X(t)+sD(t))\,\mathrm ds.
\end{equation*}
The preceding computation and
\eqref{eq:reference-right-refined-oldstyle}--\eqref{eq:reference-right-sharp-oldstyle}
imply
\begin{equation}
\label{ineq:D(t):right-quadratic-est}
 -D(t)^\top Q(t)D(t)
 \le \left(\frac{2}{1+T-t}+\Lambda_0(t)+C|D(t)|\right)|D(t)|^2,
\end{equation}
where
\begin{equation*}
 \Lambda_0(t)
 =C\left(\frac{1}{(1+T-t)^{1+\gamma}}
 +\frac{1}{(1+T-t)^2}\right),
 \qquad
 \int_0^T\Lambda_0(t)\,\mathrm dt\le C.
\end{equation*}
Indeed, at $\bar X(t)$ the leading part of $-(\nabla V)_{\rm sym}$ is
$\lambda(\bar X(t))\operatorname{Id}$ and remainder is bounded by
\begin{equation*}
  C(\bar x_{n_2}+\bar y_{n_2}^2)   \le \Lambda_0(t)
\end{equation*}
by \eqref{eq:reference-right-refined-oldstyle}. Since (by the mean value theorem)
\begin{equation*}
    \lambda \le {2c_{2,n_2}}(y_{n_2} - x_{n_2}) \le 2c_{2,n_2}y_{n_2},
\end{equation*}
we have by
\eqref{eq:reference-right-sharp-oldstyle} that
\[
 \lambda(\bar X(t))
 \le 2c_{2,n_2}\bar{y}_{n_2}(t) \le\frac{2}{1+T-t}
 +\frac{C}{(1+T-t)^{1+\gamma}}
\]

Finally, $\nabla V$ is Lipschitz on the
compact ordered endpoint region, so replacing $\bar X(t)$ by any point on the
segment joining $X(t)$ to $\bar X(t)$ costs at most $C|D(t)|$.  Integrating on
that segment proves \eqref{ineq:D(t):right-quadratic-est}.

We then estimate the remaining terms in the equation for $D$.  Since
\[
 m[g_1](t)\lesssim \frac{1}{(1+t)^2},
 \qquad
 |\mathcal S(\bar X(t))|\lesssim\frac{1}{1+T-t},
\]
and $\mathcal S$ is Lipschitz, \eqref{eqn:D(t):backwards} and \eqref{ineq:D(t):right-quadratic-est} give
\begin{equation}
\label{eq:right-difference-scalar-oldstyle}
 \frac{\mathrm d}{\mathrm dt}|D(t)|
 \le \left(\frac{2}{1+T-t}+\Lambda_0(t)
 +\frac{C}{(1+t)^2}\right)|D(t)|
 +C|D(t)|^2+\frac{C}{1+T-t}
 |m[g_1](t)-m[\bar g_1](t)|.
\end{equation}

Set
\[
 E(t)=(1+T-t)^2|D(t)|.
\]
The derivative of the weight cancels the term $2(1+T-t)^{-1}|D|$ in
\eqref{eq:right-difference-scalar-oldstyle}.  Hence
\begin{equation}
\label{eq:weighted-right-difference-oldstyle}
 E'(t)\le \Lambda_1(t)E(t)
 +\frac{C}{(1+T-t)^2}E(t)^2
 +C(1+T-t)|m[g_1](t)-m[\bar g_1](t)|,
 \qquad \int_0^T\Lambda_1(t)\,\mathrm dt\le C.
\end{equation}
Moreover, as in the last part of the proof of \Cref{lem:g_1-barg_1},
\begin{equation}
\label{eq:left-moment-difference-oldstyle}
 \begin{aligned}
 |m[g_1](t)-m[\bar g_1](t)|
 &\lesssim \frac{1}{1+t}\,
 \mathrm d(g_1(t),\bar g_1(t))\\
 &\lesssim \frac{1}{1+t}
 \int_0^t\frac{M[g_2](\sigma)}{1+\sigma}\,\mathrm d\sigma.
 \end{aligned}
\end{equation}
The first inequality follows by integrating $\sin(2\theta)$ over corresponding
endpoint intervals and using that every endpoint of $g_1$ and $\bar g_1$ is
$O((1+t)^{-1})$; the second is \Cref{lem:g_1-barg_1}.

Fix $0<\tau\le T$. Since $D(0)=0$, Gr\"{o}nwall,
\eqref{eq:weighted-right-difference-oldstyle} and
\eqref{eq:left-moment-difference-oldstyle} yield, for $0\le t\le\tau$,
\begin{equation}
\label{eq:weighted-bihari-oldstyle}
 E(t)\le C\mathcal F_T(t)
 +C\int_0^t\frac{E(s)^2}{(1+T-s)^2}\,\mathrm ds.
\end{equation}
Since $\int_0^\tau(1+T-s)^{-2}\,\mathrm ds\le1$, a standard continuation
argument shows that, if $\mathcal F_T(\tau)\le\varepsilon_0$ and
$\varepsilon_0$ is sufficiently small, then
\[
 E(t)\le C\mathcal F_T(t),\qquad 0\le t\le\tau.
\]
This proves \eqref{eq:g2-weighted-stability}.  The $L^1$ estimate follows
exactly as in \Cref{lem:g_1-barg_1}, because corresponding step functions have
the same heights and topology.

It remains to verify the last assertion.  Under
\eqref{eq:g2-moment-bootstrap-lemma}, since
\begin{equation}
\label{eq:inner-integral-bound-oldstyle}
 \int_0^s\frac{\mathrm d\sigma}
 {(1+\sigma)(1+T-\sigma)^2}
 \le C\left(\frac{\log (2+T)}{(2+T)^2}
 +\frac{1}{(2+T)(1+T-s)}\right),
\end{equation}
we have
\begin{align*}
 \mathcal F_T(T_*)
 &\le CC_0\int_0^{T_*}\frac{1+T-s}{1+s}
 \left(\frac{\log (2+T)}{(2+T)^2}
 +\frac{1}{(2+T)(1+T-s)}\right)\mathrm ds\\
 &\le CC_0\int_0^T\frac{1+T-s}{1+s}
 \left(\frac{\log (2+T)}{(2+T)^2}
 +\frac{1}{(2+T)(1+T-s)}\right)\mathrm ds\\
 &\le CC_0\frac{\log^2(2+T)}{2+T}.
\end{align*}
For $T$ sufficiently large this is smaller than $\varepsilon_0$. Applying the local estimate above with $\tau=T_*$ gives
\eqref{eq:g2-bootstrap-stability}.
\end{proof}

\section{Construction of $g_0$}
In this section, we shall construct $g_0$ and hence prove \Cref{thm:main}. 
\subsection{Finitely many targets}
The following statement extends \Cref{lem:insertion} to the case of $n$ target step functions. 
\begin{lemma}
\label[lemma]{lem:n-targets}
Let $h_1,h_2,\dots$ be a finite or infinite sequence of step functions supported in $(0,\frac{\pi}{m})$ and satisfying $0\le h_j\le1$. Let $\ve_j,\delta_j>0$. Then there exist times
\[
0=t_1<t_2<\cdots
\]
and solutions $g^{(n)}$ of \eqref{eqn:Euler:reduced:transport}-\eqref{eqn:Euler:reduced:BSlaw} with the following properties.
\begin{enumerate}
    \item For each $n$,
    \[
        g^{(n)}(0)=\sum_{j=1}^n \bar g_j(0),
    \]
    where the $\bar g_j(0)$ are step functions with disjoint and ordered supports. In particular, writing $g_j^{(n)}(t)$ for their transports under the flow of $g^{(n)}$,
    \[
        g^{(n)}(t,\theta)=\sum_{j=1}^n g_j^{(n)}(t,\theta),
        \qquad
        \sup\operatorname{supp}g_j^{(n)}(t)<\inf\operatorname{supp}g_{j+1}^{(n)}(t).
    \]
    \item
    \[
        \|g^{(n)}(t_n)-h_n\|_{L^1}\le \ve_n.
    \]
    \item 
    \begin{equation}
    \label{eqn:recursive-stage-closeness}
        \sup_{0\le t\le t_n}\|g^{(n+1)}(t)-g^{(n)}(t)\|_{L^1}\le \delta_n.
    \end{equation}
\end{enumerate}
Consequently, for $1\le j\le n$,
\begin{equation}
\label{eqn:recursive-target-error}
    \|g^{(n)}(t_j)-h_j\|_{L^1}
    \le \ve_j+\sum_{k=j}^{n-1}\delta_k.
\end{equation}
\end{lemma}

\begin{proof}
Set $t_1=0$, $g^{(1)}(0)=h_1$, and $\bar g_1(0)=h_1$. Suppose that $g^{(n)}$ and $t_1<\cdots<t_n$ have been constructed. We apply \Cref{lem:insertion} with $h_1=g^{(n)}(0)$ and $h_2=h_{n+1}$, and write the resulting solution as
\[
    g^{(n+1)}=g_1+g_2.
\]
In the notation of \Cref{lem:insertion}, $\bar g_1=g^{(n)}$. By taking $t_{n+1}$ sufficiently large, we have
\[
    \|g^{(n+1)}(t_{n+1})-h_{n+1}\|_{L^1}\le \ve_{n+1}.
\]
Moreover, for $0\le t\le t_n$, the estimates from  \Cref{lem:insertion} give
\begin{align*}
    \|g^{(n+1)}(t)-g^{(n)}(t)\|_{L^1}
    &\le \|g_1(t)-g^{(n)}(t)\|_{L^1}+\|g_2(t)\|_{L^1}\\
    &\lesssim \frac{\log^2(1+t_{n+1})}{1+t_{n+1}}
       +\frac{1}{1+t_{n+1}-t}\\
    &\lesssim \frac{\log^2(1+t_{n+1})}{1+t_{n+1}}
       +\frac{1}{1+t_{n+1}-t_n}.
\end{align*}
Thus, increasing $t_{n+1}$ if necessary, we obtain \eqref{eqn:recursive-stage-closeness} and may also impose any prescribed lower bound on $t_{n+1}$. Since $g_1(0)=g^{(n)}(0)$ and $g_2(0)=\bar g_2(0)$ in the construction of \Cref{lem:insertion}, setting $\bar g_{n+1}(0):=g_2(0)$ gives
\[
    g^{(n+1)}(0)=g^{(n)}(0)+\bar g_{n+1}(0),
\]
with $\bar g_{n+1}(0)$ lying strictly to the right of $g^{(n)}(0)$. This proves that
    \[
        g^{(n)}(0)=\sum_{j=1}^n \bar g_j(0).
    \]

Finally, for $1\le j\le n$, telescoping and \eqref{eqn:recursive-stage-closeness} give
\begin{align*}
    \|g^{(n)}(t_j)-h_j\|_{L^1}
    &\le \|g^{(j)}(t_j)-h_j\|_{L^1}
      +\sum_{k=j}^{n-1}\|g^{(k+1)}(t_j)-g^{(k)}(t_j)\|_{L^1}\\
    &\le \ve_j+\sum_{k=j}^{n-1}\delta_k,
\end{align*}
which proves \eqref{eqn:recursive-target-error}.
\end{proof}

\subsection{Infinitely many targets}

\begin{lemma}
\label[lemma]{lem:infty-targets}
Let $h_1,h_2,\dots$ be a sequence of step functions supported in $(0,\frac{\pi}{m})$ and satisfying $0\le h_j\le1$. There exist times $t_1<t_2<\cdots$ tending to infinity and a solution $g$ of \eqref{eqn:Euler:reduced:transport}-\eqref{eqn:Euler:reduced:BSlaw} such that
\[
 \|g(t_j)-h_j\|_{L^1}\le 4\cdot2^{-j},\qquad j=1,2,\dots.
\]
Moreover, the initial datum is a sum of step functions with disjoint ordered supports,
\[
 g(0,\theta)=\sum_{j=1}^\infty {\bar{g}}_j(\theta),
 \qquad \sup\operatorname{supp}{\bar{g}}_j<\inf\operatorname{supp}{\bar{g}}_{j+1},
 \qquad 0\le g(0,\theta)\le1,
\]
and, writing $g_j(t)$ for the transport of ${\bar{g}}_j(0)$ by the flow of $g$, one has
\[
 g(t,\theta)=\sum_{j=1}^\infty g_j(t,\theta)
\]
with disjoint ordered supports for every $t\ge0$.
\end{lemma}
\begin{proof}
Apply \Cref{lem:n-targets} with
\[
    \ve_n=2^{-n},\qquad \delta_n=2^{-n}.
\]
From the proof of \Cref{lem:n-targets}, we obtain times $t_n\to\infty$ and solutions $g^{(n)}$ such that
\begin{equation}
\label{eqn:infinite-stage-estimates}
    \|g^{(n)}(t_n)-h_n\|_{L^1}\le2^{-n},
    \qquad
    \sup_{0\le t\le t_n}\|g^{(n+1)}(t)-g^{(n)}(t)\|_{L^1}\le2^{-n},
\end{equation}
and
\[
    g^{(n)}(0)=\sum_{j=1}^n\bar g_j(0),
    \qquad
    \sup\operatorname{supp}\bar g_j(0)<\inf\operatorname{supp}\bar g_{j+1}(0).
\]
It follows from \eqref{eqn:infinite-stage-estimates} that $g^{(n)}$ is Cauchy in $C([0,T];L^1)$ for every $T>0$. Let $g$ denote the limit. At time $0$,
\[
    g_0(\theta)=\sum_{j=1}^\infty \bar g_j(0,\theta),
\]
where the summands have disjoint and ordered supports, and $0\le g_0\le1$.

For the corresponding velocities,
\[
    G^{(n)}(t,\theta)=\int_0^{\frac{\pi}{m}}K_m(\theta,\zeta)g^{(n)}(t,\zeta)\,\mathrm d\zeta,
\]
we have $G^{(n)}\to G$ in $C([0,T];C^1)$ for every $T>0$. Passing to the limit in the weak formulation shows that $g$ solves \eqref{eqn:Euler:reduced:transport}-\eqref{eqn:Euler:reduced:BSlaw} with initial data $g_0$.

Let $\Phi_t$ denote the characteristic flow of $g$. If $g_j$ is the transport of $\bar g_j(0)$ by this flow, then the order-preserving property of $\Phi_t$ gives
\[
    g(t,\theta)=\sum_{j=1}^\infty g_j(t,\theta),
    \qquad
    \sup\operatorname{supp}g_j(t)<\inf\operatorname{supp}g_{j+1}(t),
\]
for every $t\ge0$. In particular, $0\le g\le1$.

Finally, for every $j$,
\begin{align*}
    \|g(t_j)-h_j\|_{L^1}
    &\le \|g^{(j)}(t_j)-h_j\|_{L^1}
       +\|g(t_j)-g^{(j)}(t_j)\|_{L^1}\\
    &\le 2^{-j}+\sum_{k=j}^\infty2^{-k}
     \le 4\cdot2^{-j}.
\end{align*}
This concludes the proof.
\end{proof}

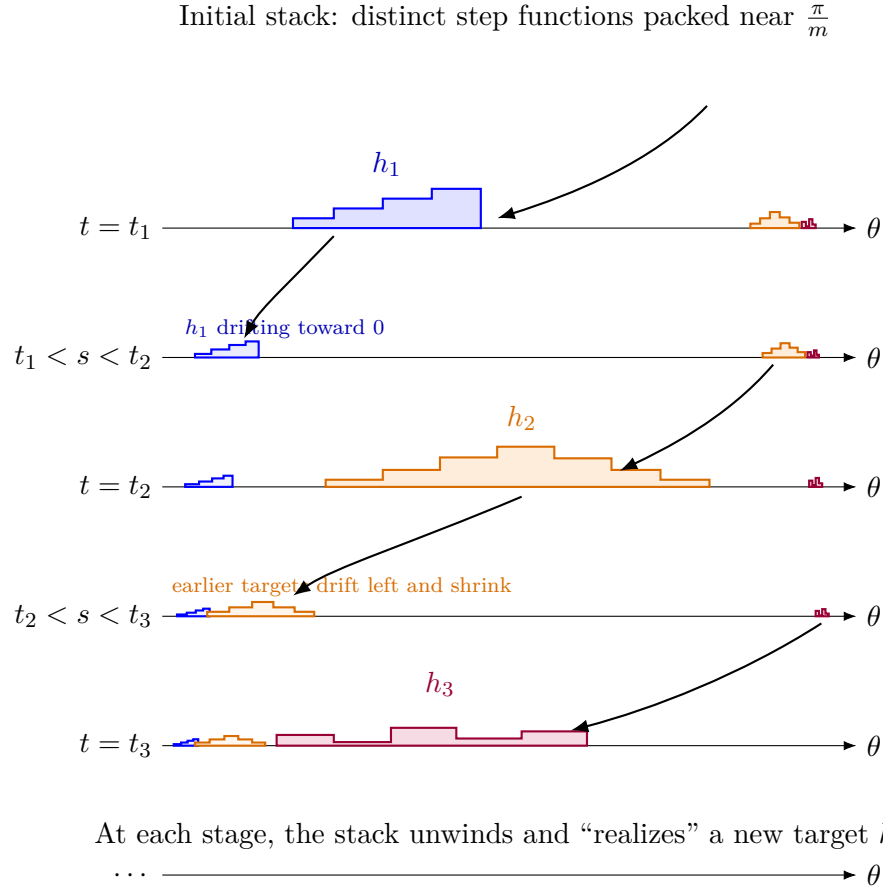
\begin{figure}[p]
\centering
\begin{tikzpicture}[x=10.8cm,y=1.18cm,>=Latex]

  \node at (0.42,0.88)
  {Initial stack: distinct step functions packed near $\frac{\pi}{m}$};

  \draw[->] (0,-1.45) -- (0.85,-1.45) node[right] {$\theta$};
  \node[left] at (0,-1.45) {$t=t_1$};

  \draw[fill=blue!12, draw=blue, thick]
    (0.16,-1.45) -- (0.16,-1.34) -- (0.21,-1.34)
    -- (0.21,-1.23) -- (0.27,-1.23)
    -- (0.27,-1.12) -- (0.33,-1.12)
    -- (0.33,-1.01) -- (0.39,-1.01)
    -- (0.39,-1.45) -- cycle;
  \node[blue!70!black] at (0.275,-0.73) {$h_1$};

  \draw[fill=orange!14, draw=orange!85!black, thick]
    (0.72,-1.45) -- (0.72,-1.40) -- (0.732,-1.40)
    -- (0.732,-1.34) -- (0.744,-1.34)
    -- (0.744,-1.27) -- (0.756,-1.27)
    -- (0.756,-1.33) -- (0.768,-1.33)
    -- (0.768,-1.39) -- (0.780,-1.39)
    -- (0.780,-1.45) -- cycle;

  \draw[fill=purple!14, draw=purple!80!black, thick]
    (0.783,-1.45) -- (0.783,-1.38) -- (0.788,-1.38)
    -- (0.788,-1.43) -- (0.792,-1.43)
    -- (0.792,-1.35) -- (0.796,-1.35)
    -- (0.796,-1.41) -- (0.800,-1.41)
    -- (0.800,-1.45) -- cycle;

  \draw[->] (0,-2.90) -- (0.85,-2.90) node[right] {$\theta$};
  \node[left] at (0,-2.90) {$t_1<s<t_2$};

  \draw[fill=blue!8, draw=blue, thick]
    (0.04,-2.90) -- (0.04,-2.86) -- (0.06,-2.86)
    -- (0.06,-2.81) -- (0.082,-2.81)
    -- (0.082,-2.76) -- (0.102,-2.76)
    -- (0.102,-2.72) -- (0.118,-2.72)
    -- (0.118,-2.90) -- cycle;
  \node[blue!70!black, font=\scriptsize] at (0.15,-2.58)
    {$h_1$ drifting toward $0$};

  \draw[fill=orange!14, draw=orange!85!black, thick]
    (0.735,-2.90) -- (0.735,-2.85) -- (0.746,-2.85)
    -- (0.746,-2.80) -- (0.757,-2.80)
    -- (0.757,-2.74) -- (0.768,-2.74)
    -- (0.768,-2.79) -- (0.778,-2.79)
    -- (0.778,-2.84) -- (0.787,-2.84)
    -- (0.787,-2.90) -- cycle;

  \draw[fill=purple!14, draw=purple!80!black, thick]
    (0.790,-2.90) -- (0.790,-2.84) -- (0.794,-2.84)
    -- (0.794,-2.88) -- (0.798,-2.88)
    -- (0.798,-2.82) -- (0.801,-2.82)
    -- (0.801,-2.87) -- (0.804,-2.87)
    -- (0.804,-2.90) -- cycle;

  \draw[->] (0,-4.35) -- (0.85,-4.35) node[right] {$\theta$};
  \node[left] at (0,-4.35) {$t=t_2$};

  \draw[fill=blue!6, draw=blue, thick]
    (0.028,-4.35) -- (0.028,-4.32) -- (0.045,-4.32)
    -- (0.045,-4.29) -- (0.061,-4.29)
    -- (0.061,-4.255) -- (0.075,-4.255)
    -- (0.075,-4.225) -- (0.086,-4.225)
    -- (0.086,-4.35) -- cycle;

  \draw[fill=orange!14, draw=orange!85!black, thick]
    (0.20,-4.35) -- (0.20,-4.27) -- (0.27,-4.27)
    -- (0.27,-4.16) -- (0.34,-4.16)
    -- (0.34,-4.02) -- (0.41,-4.02)
    -- (0.41,-3.90) -- (0.48,-3.90)
    -- (0.48,-4.03) -- (0.55,-4.03)
    -- (0.55,-4.16) -- (0.61,-4.16)
    -- (0.61,-4.27) -- (0.67,-4.27)
    -- (0.67,-4.35) -- cycle;
  \node[orange!85!black] at (0.44,-3.58) {$h_2$};

  \draw[fill=purple!14, draw=purple!80!black, thick]
    (0.792,-4.35) -- (0.792,-4.28) -- (0.796,-4.28)
    -- (0.796,-4.33) -- (0.800,-4.33)
    -- (0.800,-4.25) -- (0.804,-4.25)
    -- (0.804,-4.31) -- (0.808,-4.31)
    -- (0.808,-4.35) -- cycle;

  \draw[->] (0,-5.80) -- (0.85,-5.80) node[right] {$\theta$};
  \node[left] at (0,-5.80) {$t_2<s<t_3$};

  \draw[fill=blue!5, draw=blue, thick]
    (0.018,-5.80) -- (0.018,-5.78) -- (0.030,-5.78)
    -- (0.030,-5.76) -- (0.041,-5.76)
    -- (0.041,-5.735) -- (0.050,-5.735)
    -- (0.050,-5.715) -- (0.058,-5.715)
    -- (0.058,-5.80) -- cycle;

  \draw[fill=orange!8, draw=orange!85!black, thick]
    (0.055,-5.80) -- (0.055,-5.75) -- (0.082,-5.75)
    -- (0.082,-5.70) -- (0.110,-5.70)
    -- (0.110,-5.64) -- (0.136,-5.64)
    -- (0.136,-5.70) -- (0.162,-5.70)
    -- (0.162,-5.75) -- (0.186,-5.75)
    -- (0.186,-5.80) -- cycle;

  \node[orange!85!black, font=\scriptsize] at (0.22,-5.47)
    {earlier targets drift left and shrink};

  \draw[fill=purple!14, draw=purple!80!black, thick]
    (0.800,-5.80) -- (0.800,-5.74) -- (0.804,-5.74)
    -- (0.804,-5.79) -- (0.808,-5.79)
    -- (0.808,-5.72) -- (0.812,-5.72)
    -- (0.812,-5.77) -- (0.816,-5.77)
    -- (0.816,-5.80) -- cycle;

  \draw[->] (0,-7.25) -- (0.85,-7.25) node[right] {$\theta$};
  \node[left] at (0,-7.25) {$t=t_3$};

  \draw[fill=blue!5, draw=blue, thick]
    (0.014,-7.25) -- (0.014,-7.233) -- (0.023,-7.233)
    -- (0.023,-7.215) -- (0.031,-7.215)
    -- (0.031,-7.195) -- (0.038,-7.195)
    -- (0.038,-7.178) -- (0.044,-7.178)
    -- (0.044,-7.25) -- cycle;

  \draw[fill=orange!6, draw=orange!85!black, thick]
    (0.040,-7.25) -- (0.040,-7.215) -- (0.058,-7.215)
    -- (0.058,-7.180) -- (0.076,-7.180)
    -- (0.076,-7.142) -- (0.093,-7.142)
    -- (0.093,-7.180) -- (0.110,-7.180)
    -- (0.110,-7.215) -- (0.126,-7.215)
    -- (0.126,-7.25) -- cycle;

  \draw[fill=purple!14, draw=purple!80!black, thick]
    (0.14,-7.25) -- (0.14,-7.13) -- (0.21,-7.13)
    -- (0.21,-7.21) -- (0.28,-7.21)
    -- (0.28,-7.05) -- (0.36,-7.05)
    -- (0.36,-7.17) -- (0.44,-7.17)
    -- (0.44,-7.09) -- (0.52,-7.09)
    -- (0.52,-7.25) -- cycle;
  \node[purple!80!black] at (0.34,-6.56) {$h_3$};

  \draw[->] (0,-8.70) -- (0.85,-8.70) node[right] {$\theta$};
  \node[left] at (0,-8.70) {$\cdots$};
  \node at (0.42,-8.28)
  {At each stage, the stack unwinds and ``realizes" a new target $h_j$.};

  \draw[->, thick]
    (0.667,-0.08) .. controls (0.60,-0.72) and (0.50,-1.08) .. (0.41,-1.34);

  \draw[->, thick]
    (0.748,-2.98) .. controls (0.69,-3.60) and (0.61,-3.95) .. (0.56,-4.17);

  \draw[->, thick]
    (0.807,-5.88) .. controls (0.71,-6.45) and (0.60,-6.84) .. (0.50,-7.08);

  \draw[->, thick]
    (0.21,-1.54) .. controls (0.16,-2.02) and (0.12,-2.35) .. (0.10,-2.69);

  \draw[->, thick]
    (0.44,-4.46) .. controls (0.30,-5.00) and (0.21,-5.25) .. (0.16,-5.57);
\end{tikzpicture}
\caption{A  schematic for the infinite-target mechanism in Section~4. Initially one prepares a stack of step functions accumulating near $\frac{\pi}{m}$. As time advances, one step function after another unwinds into the bulk and realizes (approximately) the next target, while the step functions intended for later times remain compressed near the right endpoint.}
\label{fig:recursive-construction-schematic}
\end{figure}

\subsection{Proof of the main theorem}
We are now poised to prove \Cref{thm:main}. It is enough to prove the following theorem on $[0,\frac{\pi}{m}]$. Indeed, on uniformly bounded subsets of $L^\infty$, $L^1$ convergence implies weak$^*$ convergence. We may then extend the solution to $\Ss^1$ by odd symmetry and $m$-fold symmetry. The symmetry, sign and $L^\infty$ conditions are preserved by the flow and weak$^*$ closed, which gives the reverse inclusion in \Cref{thm:main}. 
\begin{theorem}Let $m \ge 4$. There exists initial data $g_0 \in L^\infty([0, \frac{\pi}{m}])$, such that the $L^1([0, \frac{\pi}{m}])$ $\omega_+$-limit set of $g_0$ is
    \begin{equation*}
       \omega_+(g_0) = \{h \in L^\infty([0, {\textstyle\frac{\pi}{m}}]): 0 \le h \le \|g_0\|_{L^\infty}\}
    \end{equation*}
\end{theorem}

\begin{proof}
Let
\[
S=\{f\in L^\infty([0,{\textstyle\frac{\pi}{m}}]): 0 \le f\le1\}.
\]
The countable set of step functions
\[
\sum_{k=1}^n c_k\mathbf 1_{(a_k,b_k]}(\theta),
\]
where
\[
0<a_1<b_1<\cdots<a_n<b_n<\frac{\pi}{m},
\qquad a_k,b_k,c_k\in\mathbb Q,\qquad 0<c_k\le1,
\]
is dense in $S$ in $L^1$. Let $h_1,h_2,\dots$ be an enumeration of this set. By reordering if necessary, we may assume that
\[
h_1=\mathbf 1_{(a,b]}
\]
for some rational $0<a<b<\frac{\pi}{m}$. Notice also that this countable dense set has no isolated points in $L^1$; for example, one may perturb a rational coefficient or endpoint by an arbitrarily small rational amount. Thus, for every $f\in S$, there is a subsequence $h_{n_k}$ with $n_k\to\infty$ and $h_{n_k}\to f$ in $L^1$.

By \Cref{lem:infty-targets}, there exists a solution $g$ and times $t_1<t_2<\cdots$ tending to infinity such that
\[
\|g(t_j)-h_j\|_{L^1}\le4\cdot2^{-j}.
\]
Moreover, by the construction in \Cref{lem:n-targets}, $\bar g_1(0)=h_1$. Hence $0\le g_0\le1$ and
\[
\|g_0\|_{L^\infty}=1.
\]

Let $f\in S$. Then
\begin{align*}
\|g(t_{n_k})-f\|_{L^1}
&\le \|g(t_{n_k})-h_{n_k}\|_{L^1}+\|h_{n_k}-f\|_{L^1}\\
&\le4\cdot2^{-n_k}+\|h_{n_k}-f\|_{L^1}\longrightarrow0.
\end{align*}
Thus $S\subset\omega_+(g_0)$. Since transport preserves $0\le g(t)\le1$ and $S$ is closed in $L^1$, we also have $\omega_+(g_0)\subset S$. This concludes the proof.
\end{proof}

\section*{Acknowledgments and AI tool disclosure}
The idea for \Cref{thm:main} emerged from discussions with Tarek Elgindi. The author is grateful to him for helpful discussions and for comments on an earlier draft of this paper. ChatGPT was used for proofreading and to help identify and correct technical inaccuracies in an initial draft of this paper. The figures in this paper were also generated with its assistance. The core proof idea and architecture were developed by the author. Outside of the tool use described, this paper is human generated.

{\small
\bibliographystyle{plain}
\bibliography{refs.bib}
}

\end{document}